\documentclass[english,11pt]{smfartVScomptage}

\usepackage{smfenum,smfthm, amsmath,amssymb,amscd,mathrsfs,euscript,color}
\usepackage{stmaryrd,mathabx,upgreek} 
\usepackage[latin1]{inputenc}
\usepackage[T1]{fontenc}
\usepackage{xcolor}
\input{xypic}

\numberwithin{equation}{section} 
\def\AA{\mathfrak{A}}

\def\CC{\mathbb{C}}

\def\QQ{\mathbb{Q}}

\def\D{{\rm D}}

\def\F{{ F}}
\def\G{{ G}}
\def\H{{ H}}
\def\I{{ I}}
\def\J{{\rm J}}

\def\L{{\rm L}}

\def\N{{ N}}
\def\O{{\rm O}}
\def\P{{ P}}

\def\R{{ R}}

\def\U{{ U}}

\def\W{{ W}}
\def\X{{ X}}

\def\Cc{\EuScript{C}}

\def\Ii{I}

\def\Ww{\EuScript{W}}

\def\Ga{\Gamma}
\def\La{{\it\Lambda}}
\def\Lat{\La^*}

\def\a{\alpha} 

\def\d{\delta}
\def\e{\varepsilon}

\def\g{\gamma}

\def\k{\boldsymbol{k}}
\def\l{\lambda}

\def\p{\mathfrak{p}}
\def\r{{\textbf{\textsf{r}}}}
\def\s{\sigma}
\def\t{\theta}

\def\w{\varpi}

\def\kk{\boldsymbol{k}}

\def\({\left(}
\def\){\right)}
\def\>{\geqslant}
\def\<{\leqslant}

\def\tdt{\times\cdots\times}

\def\Hom{\operatorname{Hom}}
\def\End{\operatorname{End}}

\def\Mat{\operatorname{M}}
\def\GL{\operatorname{GL}}
\def\Gal{\operatorname{Gal}}

\def\Ind{\operatorname{Ind}}

\def\dim{\operatorname{dim}}

\def\st{\operatorname{sp}}
\def\St{\operatorname{Sp}}

\def\qlb{{\overline{\mathbb{Q}}_\ell}}
\def\zlb{{\overline{\mathbb{Z}}_\ell}}
\def\flb{{\overline{\mathbb{F}}_\ell}}

\def\oo{\EuScript{O}}
\def\pp{\mathfrak{p}}

\def\bl{\boldsymbol{\lambda}}

\def\det{\mathrm{det}}

\def\rl{\textbf{{\textsf{r}}}_\ell}
\def\1{\mathbf{1}}

\def\rt{\widetilde{\rho}}

\def\FF{\mathbb{F}}
\def\WW{\EuScript{W}}

\def\qe{q}

\def\ke{\boldsymbol{k}}
\def\kf{\boldsymbol{k}_{0}}

\def\et{e}

\def\ZZ{\mathbb{Z}}

\def\Ker{{\rm Ker}}

\def\Mat{\boldsymbol{{\sf M}}}

\def\Cc{\EuScript{C}}

\def\RS{}

\def\Be{\J}
\def\coef{f}
\def\cen{\omega}
\def\cla{\eta}
\def\row{\eta}
\def\ww{w}
\def\wn{w_{n}}

\def\Vi{{\sf V}}
\def\Wi{{\sf W}}

\title[Cuspidal 
representations of $\GL_n(F)$ distinguished by a Galois involution] 
{Cuspidal $\ell$-modular representations of $\GL_n(F)$ distinguished
by a Galois involution, II} 
\author{Robert Kurinczuk}
\address{School of Mathematics and Statistics, University of Sheffield,
Sheffield, S3 7RH, United Kingdom}
\email{robkurinczuk@gmail.com}
\author{Nadir Matringe}
\address{Shanghai Institute for Mathematics and Interdisciplinary Sciences, 
  Block A,
  Internatio\-nal Innovation Plaza,
  N°657 Songhu Road,
  Yangpu District,
  Shanghai,
  China}
\email{nadirmatringe@outlook.fr}
\author{Vincent S\'echerre} 
\address{Laboratoire de Math\'emati\-ques de Versailles\\
UVSQ\\
CNRS\\
Universit\'e Paris-Saclay\\
78035, Versailles, France}
\email{vincent.secherre@uvsq.fr}

\begin{abstract} 
Let $\F/\F_0$ be a quadratic extension of non-Archimedean locally compact 
fields with residual characteristic $p\neq2$, 
and $\ell$ be a prime number different from $p$.
We classify those $\ell$-modular cuspidal irreducible representations of 
$\GL_n(F)$ which are $\GL_n(F_0)$-distinguished, 
that is,~which~car\-ry a non-zero $\GL_n(F_0)$-invariant linear form.
In the case when~$\ell\neq2$, 
an $\ell$-modular~cuspidal~repre\-sentation of $\GL_n(F)$ is 
$\GL_n(F_0)$-distinguished if and only if it lifts to a 
$\GL_n(F_0)$-distinguished~cus\-pi\-dal $\ell$-adic repre\-sentation,
whereas when $\ell=2$,
it is $\GL_n(F_0)$-distinguished if and only if it~is~con\-ju\-ga\-te-self-dual.
\end{abstract} 

\long\def\MSC#1\EndMSC{\def\arg{#1}\ifx\arg\empty\relax\else
     {\par\narrower\noindent%
     2010 Mathematics Subject Classification: #1\par}\fi}

\long\def\KEY#1\EndKEY{\def\arg{#1}\ifx\arg\empty\relax\else
	{\par\narrower\noindent Keywords and Phrases: #1\par}\fi}

\begin{document}

\maketitle

\MSC 22E50, 11F70, 20C20
% \subjclass[2010]{22E50; 11F70; 11F66}
% \hfill\textcolor{green}{\bf\today}
\EndMSC
\KEY 
Cuspidal representations,
Distinction,
Lifting,
Modular representations,
Rankin--Selberg gamma factors
\EndKEY

\section{Introduction}

\subsection{}

Let $\F/\F_0$ be a quadratic extension of non-Archimedean locally compact 
fields with residual cha\-racteristic $p\neq2$.
Fix a positive integer $n$ and set $\G=\GL_n(\F)$ and
$H=\GL_n(F_0)$.~Let~$\R$~be~an algebraically closed field of
characteristic $\ell>0$ different from $p$. 
In \cite{NRV25}, we addressed the~follow\-ing problem.

\begin{enonce}{Problem}
\label{mntcrst}
Classify the cuspidal, 
irreducible,
smooth $\R$-representations of $G$ which~are~dis\-tinguished by $H$,
that is,
which carry a non-zero $H$-invariant linear form.
\end{enonce}

In this article,
we give a complete ans\-wer to Problem \ref{mntcrst}.
The classification can be stated~in~ve\-ry simple terms. 
However,
the case where $\ell\neq2$ and the case~whe\-re~$\ell=2$
are com\-ple\-te\-ly different.
Let us explain what happens.

\subsection{}

Let us first briefly recall what happens for complex representations. 
Let $\s$~be~the non-trivial automorphism of $F/F_0$
and $\cla$ be the character of $F_0^\times$ with kernel 
the~sub\-group of $F/F_0$-norms~of $F_0^\times$.
One has the three following facts:

\medskip

(\textit{Conjugate-self-duality})
Any irreducible complex representation of $G$ distinguished by $H$
is~$\s$-self-dual,
that is,
its~con\-tra\-gre\-dient is isomorphic to its con\-jugate by $\s$ 
(\cite{Flicker} Proposition 12).

\medskip

(\textit{Dichotomy and Disjunction})
A $\s$-self-dual cuspidal complex irreducible representation $\pi$ of~$G$ is 
distinguished either by $H$ or by $\cla$
(the latter means that $\Hom_H(\pi,\cla\circ\det)$ is non-zero)
but~not both (\cite{Kable} Theorem~7, \cite{AKT} Corolla\-ry 1.6 if $F$ has
characteristic $0$,
and \cite{AKMSS} Theorem A.2, \cite{JoPJM23} if $F$ has 
characteristic~$p$\footnote{In characteristic $p$,
there is an issue in the proof of \cite{AKMSS} Theorem A.1. 
This is explained in \cite{JoPJM23} p.~303--304
and corrected by \cite{JoPJM23} Theorem 4.7.}). 

\medskip

(\textit{Parametrisation})
Cuspidal complex (irreducible) represen\-ta\-tions of $G$ distinguished by
$H$~are classified 
in terms of~their~Lang\-lands parameter (\cite{GanRag} Theorem 6.2):
they are the $\s$-self-dual~cus\-pi\-dal representations~whose~parame\-ter is 
conjugate-orthogonal in the sense of \cite{GGP2} Section 3.

\subsection{}
\label{P12}

Let us now consider representations with coefficients in $R$. 
For simplicity,
we will assume~in this introduction that $\R$ is an algebraic closure $\flb$ of
a finite field of cardinality~$\ell$.
First,
as~in the complex case,
any $H$-distinguished irreducible $\flb$-representation of $G$
is $\s$-self-dual~(\cite{VSANT19}~Theo\-rem 4.1).
Now let us focus on cuspidal representations.

Recall that over $\flb$,
the two notions of \textit{cuspidal} and \textit{supercuspidal}
representations do not~coin\-ci\-de~(\cite{Vigbook}),
contrary to the case of complex representations:
an irreducible representation of $G$ is cuspidal
(respectively, supercuspidal)
if it is not isomorphic to a subrepresentation
(respectively,~a subquotient)
of a representation induced from a proper parabolic subgroup of $G$.
Any super\-cus\-pi\-dal representation is cuspidal,
but the converse does not hold in general.

Let us also fix an algebraic closure $\qlb$ of the field $\QQ_\ell$
of $\ell$-adic numbers
and identify the~residue field of its ring of integers $\zlb$
with~$\flb$.~In this situation,
there is a notion of \textit{integral}~$\qlb$-representa\-tion
of $G$ and such a representation~can~be
\textit{reduced mod $\ell$}, 
which gives rise to a semi-simple~$\flb$-re\-presentation
of finite length of $G$
(see~\S\ref{defrlintro}~for~pre\-ci\-se definitions).
One then defines a $\qlb$-\textit{lift}~of~an irreducible
$\flb$-representation $\pi$ of $G$
as an~inte\-gral $\qlb$-represen\-ta\-tion of $G$ whose reduction
mod~$\ell$ is isomorphic to $\pi$.

Cuspidal representations of the group
$G$ behave well with respect to reduction and lifting:~the
reduction mod $\ell$ of an integral cuspidal $\qlb$-representation
is irreducible and cuspidal, 
and any~cus\-pidal $\flb$-representation has a $\qlb$-lift (which is
automatically cuspidal)
(\cite{Vigbook} III.1.1, III.5.10).~Re\-duc\-tion~mod $\ell$ thus~de\-fines 
a surjective map from 
isomorphism classes of integral cuspidal\! $\qlb$-re\-presentations of
$G$~to~iso\-morphism classes of cuspidal $\flb$-representations of $G$.

Moreover,
reduction mod $\ell$ of integral cuspidal $\qlb$-representations
preserves $H$-distinction:~the
reduction mod $\ell$ of an $H$-distinguished integral cuspidal
$\qlb$-representation is~$H$-dis\-tinguished~(\cite{KuMaAsai} Theorem 3.4). 
(Equivalently, 
any cuspidal $\flb$-representation of $G$
with~an~$H$-distinguished~$\qlb$-lift is $H$-distin\-guished.)
Conversely, 
any $H$-distinguished~supercus\-pi\-dal
representation~of~$G$~has~an~$H$-distinguished $\qlb$-lift (\cite{VSANT19}
Theorem 10.11, see~also
\cite{CLL}~Theo\-rem 3.4 for a more~general~sta\-te\-ment).
There is also a classifi\-ca\-tion of the non-supercuspidal,
cuspidal $\flb$-representations~of~$G$
ha\-ving~an $H$-dis\-tinguished $\qlb$-lift 
(\cite{NRV25} Propositions 6.1, 6.2).
We will come back to it in \S\ref{xxx}.

We show that,
surprisingly, 
assuming that $\ell\neq2$,
this Distingui\-shed Lift~Pro\-per\-ty of supercus\-pidal 
representations extends to all cuspidal representations.~Our first main theorem~is:

\begin{theo}
\label{MAINTHM}
Assume that $\ell\neq2$.
Then a cuspidal $\flb$-representation of $G$ is
$H$-distinguished if and only if it has an $H$-distinguished $\qlb$-lift. 
\end{theo}

Together with our classification of the cuspidal $\flb$-representations of $G$
having an $H$-distingui\-shed $\qlb$-lift, 
this thus gives a complete classification of all cuspidal
$\flb$-representation of $G$ distin\-guished by $H$
in the case when $\ell\neq2$.

We automatically deduce from Theorem \ref{MAINTHM} the following
Disjunction Theorem:
a cuspidal~$\flb$-representation of $G$
cannot be both dis\-tin\-gui\-shed and $\cla$-distinguished
(see Corollary \ref{disjunctionmodl}).
Obser\-ve that Dichotomy holds for supercuspidal $\flb$-representations 
(\cite{VSANT19} Theorem 10.8)
but not for~cus\-pidal ones (see Remark \ref{haydee}). 

\subsection{}

When $\ell=2$,
the Dis\-tin\-guished Lift Property of
$H$-distinguished cuspidal representations~of $G$
does not hold:
there are $H$-distinguished cuspidal representations 
with no~$H$-distinguish\-ed lift
(see Remark \ref{soitditenpassant}).
However, 
the~classifi\-cation of all $H$-distinguished
cuspidal~$\overline{\FF}_2$-representations~of $G$ 
turns out to be even easier to state.

As has been said in \S\ref{P12},
any $H$-distinguished irreducible $\overline{\FF}_2$-representa\-tion~of $G$
is $\s$-self-dual.
Conversely, 
a supercuspidal $\overline{\FF}_2$-repre\-sentation of $G$
is distinguish\-ed by $H$ if and only if it is $\s$-self-dual
(\cite{VSANT19} Theorem 10.8).
(Note that there is no 
$\overline{\FF}_2$-character~of $F_0^\times$
with kernel~the sub\-group~of $F/F_0$-norms.)

We prove that,
surprisingly again,
this characterisation of distinction 
for supercuspidal repre\-sen\-tations
ex\-tends to all $\overline{\FF}_2$-cuspidal representations,
in contrast with the case where $\ell\neq2$. Our second main theorem~is:

\begin{theo}
\label{MAINTHM2intro}
A cuspidal $\overline{\FF}_2$-representation of $G$ is 
distinguished by $H$~if~and only if it is~$\s$-self-dual.
\end{theo}

\subsection{}

Let us now explain the mains ideas of our proof of our main theorems,
starting with~Theo\-rem \ref{MAINTHM}.
{In the complex case, 
studying the distinction of a quotient of a parabolically
induced~re\-pre\-sentation is a notoriously non-trivial problem.
In our situation,
where representations have~co\-ef\-ficients in $\flb$, 
the situation is even worse:
a non-supercuspidal, cuspidal representation of~$G$~can\-not be
realized as a quotient of a proper parabolically induced
representation,
although it is~a~sub\-quotient of a proper parabolically induced
representation.
Also, the analytic tools that have been developed in the complex case
to study distinction 
are no longer available in our modular case. }

% We follow a completely different route. 
First,
thanks to the results of \cite{NRV25},
we reduce~to~the case of~$\flb$-cuspidal representations of~level $0$.
Secondly, 
our strategy is by contradiction:
we assume that there is a cuspidal~representation~$\pi$
of level $0$ which is distinguished but has no~distinguished
cuspidal lift.
We then~com\-pu\-te a~Ran\-kin--Selberg $\g$-factor
associated with $\pi$
in two different~ways. The 
first~com\-putation gives $-1$~and the second one gives $1$.
Since $-1\neq1$ in $\flb$ when $\ell\neq2$,~we get the expected~contra\-diction. 
{Note that the idea of using $\g$-factors in distinction problems 
goes back to Hakim \cite{HakimDMJ91} --
see also~Ha\-kim--Offen \cite{HakimOffen}.}

These two computations rely on very different ideas.
The first one relies on the lifting properties of the Rankin--Selberg 
$\g$-factor of \cite{KM17}
together with the classification of
$H$-distinguished~non-supercus\-pidal, cuspidal $\flb$-representations of 
$G$ of \cite{NRV25}.
The second one relies on a $\ell$-modular,~fi\-nite field variant of a 
theorem of Ok \cite{Ok}
on $H$-distinguished cuspidal complex representations~of $G$.
Let us give more details. 

\subsection{}
\label{xxx}

Let $\pi$ be a cuspidal representation of $G$.
According to \cite{MSc},
there are a divisor~$r$ of $n$ and~a supercuspidal representation $\rho$ of 
$\GL_{n/r}(F)$ such that $\pi$ is isomorphic to the unique~generic~irre\-ducible
subquotient of the induced representation
$\rho\nu^{(1-r)/2}\times\dots\times\rho\nu^{(r-1)/2}$,
where $\nu$ is the~abso\-lute value of the determinant
and $\times$ denotes normalized parabolic induction.
Since our problem is~solved for~super\-cuspidal representations,
we may and will assume that $\pi$ is not supercuspidal,
that is,~$r\>2$.~One then can form the Rankin--Selberg $\g$-factor
$\g(X,\pi,\rho^\vee,\psi)$
with respect~to~any non-trivial smooth $\flb$-character $\psi$ of $F$
(see \cite{KM17}).
This is a non-zero element of the fraction~field $\flb(X)$
which we may evaluate at $q^{-1/2}$,
where $q^{1/2}$ is a square root in $\flb$ of the~car\-dinality~$q$~of~the 
residue field of $F$.
(In the case when $F/F_0$ is unramified,
we always choose for $q^{1/2}$ the cardinality of~the residue field of $F_0$.)
A simple use of
the lifting properties of this local factor (see Proposition \ref{gammaint})
gives us 
\begin{equation}
\g(q^{-1/2},\pi,\rho^\vee,\psi) = \cen_\pi(-1)^{n/r-1}\cdot(-1)^{r}\cdot q^{n/2} 
\end{equation}
where $\cen_\pi$ is the central character of $\pi$.
Assuming that $\pi$ is distinguished but has no distinguished lift,
and using our classification of the cuspidal $\flb$-representations of $G$
having an $H$-distinguished lift (\cite{NRV25} Propositions 6.1, 6.2),
we obtain 
\begin{equation}
\label{gammamoins1}
\g(q^{-1/2},\pi,\rho^\vee,\psi) = -1
\end{equation}
for any non-trivial $\psi$ (see Proposition \ref{gammais1}).
Let us stress that this classification
is crucial~in the proof of \eqref{gammamoins1}. 
It is stated in \cite{NRV25} in terms of type theoretic 
invariants of cuspidal representations,
but it takes a simpler form for representations of level $0$
(see \S\ref{classif6162}).
Assuming that there exists~a distinguished cuspidal representation with
no distinguished lift, 
it gives some information on the order of the
cardinality of the residue field of $F_0$ mod $\ell$
(see \S\ref{mammuth})~which allows us to compute~$q^{n/2}$. % mod $\ell$.

\subsection{}
\label{jeanbaudu}

We now have to find another way of computing the $\g$-factor of
\S\ref{xxx}.
Let $\k$ denote the~residue field of $F$.
Let $\Vi$ be a cuspidal irreducible representation of $\GL_n(\k)$
and $\Wi$ be any generic~irredu\-cible representation of $\GL_m(\k)$
for some positive integer $m<n$.
Associated with any non-tri\-vial cha\-rac\-ter $\Psi$ of $\k$,
there is a $\g$-factor $\g(\Vi,\Wi,\Psi)$ defined in \cite{Mossetal}.
We show that: 
\begin{itemize}
\item 
it is compatible with reduction mod $\ell$
(see Proposition \ref{reductionofgamma}), 
\item
it is multiplicative in the second variable (see Proposition \ref{mult}). 
\end{itemize}
Now let $\overline{\pi}$ denote the type of the cuspidal representation $\pi$
of \S\ref{xxx}.
This is the representation of $\GL_n(\k)$ defined by the 
action of $\GL_n(\oo_F)$ on the vectors of $\pi$ that are
fixed by $1+\Mat_n(\p_F)$,
where $\oo_F$ and~$\p_F$ are the ring of integers of $F$ and its maximal 
ideal. 
It is irreducible and cuspidal.~Si\-milarly,
$\rho$ has a type $\overline{\rho}$,
which is a~super\-cus\-pidal, irreducible representation of $\GL_{n/r}(\k)$,
and $\overline{\pi}$ is isomorphic to the unique generic irreducible
subquotient of $\overline{\rho}\times\dots\times\overline{\rho}$.
Moreover,
if we let $\k_0$ denote the residue field of $F_0$ and write
\begin{equation}
\label{denisebaudu}
\overline{H} = 
\left\{ 
\begin{array}{ll}
\GL_n(\k_0) & \text{if $\F/\F_0$ is unramified}, \\ 
\GL_{n/2}(\k)\times\GL_{n/2}(\k) & \text{if $\F/\F_0$ is ramified},
\end{array}\right.
\end{equation}
(note that $n$ is always even in the latter case),
it follows from \cite{NRV25}
that $\overline{\pi}$ is $\overline{H}$-distinguished~with no 
$\overline{H}$-distinguished $\qlb$-lift.
Assuming that the character~$\psi$~of \S\ref{xxx} has conductor $\p_F$, 
we deduce from \cite{NienZhang}~Theo\-rem 3.11 that
\begin{equation}
\g(q^{-1/2},\pi,\rho^\vee,\psi) = \g(\overline{\pi},\overline{\rho}^\vee,\Psi)
\end{equation}
where $\Psi$ is the character of $\k$ induced by the restriction
of $\psi$ to $\oo_F$.
We are thus~reduced to~com\-pu\-ting 
$\g(\overline{\pi},\overline{\rho}^\vee,\Psi)$, 
which entirely pertains to the theory of 
representations of finite general linear groups.

\subsection{}

We now change the notation:
for any integer $n\>1$,
we set $G_n=\GL_n(\k)$ and let~$H_n$~deno\-te~the subgroup defined by 
\eqref{denisebaudu} (we actually use a non-standard Levi subgroup of
$\GL_n(\k)$ in the Levi case, 
and also need to define $H_n$ in the case when $n$
is odd:~see Section~\ref{luigivampa}).~Let us
fix~a~non-supercuspidal, cuspidal~$\flb$-representation $\pi$ of $G_n$.
It occurs as~the unique generic~irreducible~sub\-quo\-tient of the
induced~re\-pre\-sentation $\rho\times\dots\times\rho$
for some supercuspidal $\flb$-re\-pre\-sentation $\rho$ of 
$G_{n/r}$ for some divisor $r$ of $n$.
Let $\psi$ be the character 
denoted $\Psi$~in \S\ref{jeanbaudu}.~We~prove that
\begin{equation}
\label{marmithon}
\g({\pi},{\rho}^\vee,\psi) = 1
\end{equation}
when $\pi$ is $H_n$-distinguished with no $H_n$-distinguished lift
(see Proposition \ref{minelli}).
Note that~even
if this $\g$-factor does not depend on the choice of $\psi$,
it is convenient to compute it for a specific~choi\-ce of $\psi$,
namely,
a $\psi$ trivial on $\k_0$ when $H_n=\GL_n(\k_0)$,
which we assume in Section~\ref{luigivampa}.

First,
associated with $\pi$, 
there is a scalar
$c(\pi,\psi) \in \overline{\mathbb{F}}{}_\ell^\times$,
which is a proportionality constant~bet\-ween two explicit non-zero
$H_n$-invariant 
linear forms on the Whittaker model of $\pi$ (see Proposi\-tion 
\ref{BernsteinConstant}). 
Secondly,
we prove that,
for any $H_{n-1}$-distinguished representation $\pi'$ of $G_{n-1}$
of~Whit\-ta\-ker type (see Definition \ref{defwtg})
satisfying a technical condition \eqref{saintepelagie}
(see Definition \ref{WTspecialH}),
one has
\begin{equation*}
\g(\pi,\pi',\psi) = c(\pi,\psi). 
\end{equation*}
In particular,
this $\g$-factor does not depend on $\pi'$. 
We call a distinguished representation of~Whit\-taker type 
satisfying \eqref{saintepelagie} \textit{special}. 
We then show that the class of special representations is~large enough:
it contains all distinguished cuspidal irreducible representations
and is stable under~pa\-ra\-bolic induction
(see Corollary \ref{CpsiCusp} and Lemma \ref{cuspsatcpsi2}). 
This allow us to prove that
\begin{equation}
\g({\pi},{\rho}^\vee,\psi) = \g(\pi,\psi)^{n/r}
\end{equation}
where $\g(\pi,\psi)$ is the Godement--Jacquet $\g$-factor of $\pi$.
We then use \cite{NRVGJ},
where we have computed $\g(\pi,\psi)$
(see Proposition \ref{valueGJdist}).
The expected result \eqref{marmithon} follows.

This strategy for~com\-pu\-ting $\g({\pi},{\rho}^\vee,\psi)$, 
which relies on the introduction of the constant~$c(\pi,\psi)$
is reminiscent of Ok \cite{Ok}.
However, Ok's strategy cannot be directly adapted to our modular~set\-ting. 
The main novelty in~our~approach is to introduce the class $\Cc(H)$
(see Definition \ref{Renal}),
which gives a sufficient condition for being~spe\-cial, 
and to check that it contains enough re\-pre\-sen\-tations to make 
Ok's strategy work over finite fields. 

\subsection{}

Let us now explain the main ideas of our proof of Theorem \ref{MAINTHM2intro}. 
It is based on the following observation:
if $\pi$ is a non-supercuspidal,
cuspidal $\overline{\mathbb{F}}_2$-representation of $\GL_n(\k)$,
then $n$ is even~and 
there is a unique cuspidal $\overline{\mathbb{F}}_2$-representation $\tau$
of $\GL_{n/2}(\k)$ such that $\pi$ is the unique generic~sub\-quo\-tient
of the induced representation $\tau\times\tau$. 
By uniqueness,
$\pi$ is $\s$-self-dual if and only if~$\tau$~is~$\s$-self-dual. 
Since Theorem \ref{MAINTHM2intro} is known to hold
for supercuspidal representations,
it thus suf\-fi\-ces~to prove that,
if $\tau$ is distinguished,
then $\pi$ is distinguished. 

For this,
we examine the induced representation $\tau\times\tau$.
It is indecomposable of length~$3$.~As\-suming that $\tau$ is distinguished,
we construct two independent
invariant linear forms on the~repre\-sentation $\tau\times\tau$~vani\-shing on
its socle.
If $\pi$ were not distinguished,
one would thus get two~inde\-pendent
inva\-riant $\overline{\mathbb{F}}_2$-linear forms
on the cosocle of $\tau\times\tau$.
This contradicts Proposition \ref{onedim},~which
says that the dimension of the space of invariant linear forms 
on this cosocle 
is at most $1$.
Theo\-rem~\ref{MAINTHM2intro} follows. 

\subsection{}

Let us make a couple of comments.
% Given the simplicity of the statement of Theorem~\ref{MAINTHM},
% it is tempting to look for a more conceptual proof.
% However,
% such a proof,
% if it exists,
% should~ma\-ke apparent that Theorem~\ref{MAINTHM} does not hold
% when $\ell=2$.
% Thus it should not be uniform in $\ell$.
It is tempting to conjecture that Theorem \ref{MAINTHM} holds~in
greater generality.
For~other involutions of the group $\GL_n(F)$,
or more generally for involutions~of inner forms of $\GL_n(F)$,
there is~some hope indeed.
It seems plausible as well that a result~simi\-lar to Theorem 
\ref{MAINTHM2intro} 
holds for any~involu\-tion of any inner form of $\GL_n(F)$. 
Antonin Casel is~investi\-ga\-ting these questions in his ongoing PhD thesis at
the~Uni\-versity of Versailles.

For reductive $p$-adic groups other than general linear groups,
an obstruction to a naive genera\-li\-sation of Theorem \ref{MAINTHM} 
is already apparent in the group case:
given a connected reductive $p$-adic group $H$,
diagonally embedded in $G=H\times H$,
and an irreducible $\flb$-representation $\pi$ of $H$, 
the fact that $\pi\otimes\pi^\vee$ lifts to an irreducible 
$\qlb$-representation of $G$ distinguished by $H$ is equivalent~to the fact
that $\pi$ has a $\qlb$-lift.
But this latter fact does not hold in general
(see for instance~\cite{RobANT14}~Re\-mark~5.5 for unramified unitary
groups in three variables,
which have cuspidal $\flb$-representations which do not lift to $\qlb$).

\section*{Acknowledgements}

Robert Kurinczuk was supported by the Engineering and Physical Sciences 
Research Council (EPSRC) grant EP/V001930/1. 
Nadir Matringe was supported by the Research Start-up Fund of the Shanghai 
Institute for Mathematics and Interdisciplinary Sciences (SIMIS). 
Vincent Sécherre was partly supported by the Institut Universitaire de France.

We thank Alberto M\'\i nguez and Shaun Stevens for their interest in this project.

\section{Notation}

\subsection{}
\label{par21}

Given
a locally compact, totally disconnected topological group $\G$ and
an algebraically~clo\-sed field $\R$, 
we consider smooth representations of $\G$ on $\R$-vec\-tor~spa\-ces. 
We abbreviate \textit{smooth $\R$-representation} to 
\textit{$\R$-representation},
or even \textit{represen\-ta\-tion} when $\R$ is clear from the context.

An $\R$-\textit{character} 
(or \textit{character})
of $\G$ is a group homomorphism from $\G$ to $\R^\times$ with open kernel.

Given a representation $\pi$ and a character $\chi$ of $G$, 
we write $\pi^\vee$ for the~con\-tra\-gredient of~$\pi$~and
$\pi\chi$ for the representation $g\mapsto\chi(g)\pi(g)$ of~$\G$.
If $\pi$ has a central character,
we denote it by $\cen_{\pi}$.

Given a vector $v$ in the space of $\pi$ and a linear form $\xi$ in the space
of $\pi^\vee$, we write $\coef_{v,\xi}$~for the 
matrix coefficient $g \mapsto \xi(\pi(g)v)$ on $G$.

Given any closed subgroup $H$ of $G$ and any character $\mu$ of $H$,
we say that a representation~$\pi$~of $G$ is $\mu$-\textit{distinguished} 
if the space
$\Hom_{\H}(\pi,\mu)$ is non-zero.
If $\mu$ is the trivial character,
we~abbre\-via\-te $\mu$-\textit{distinguished} to $\H$-\textit{distinguished},
or just \textit{distinguished} when $\H$ is clear from the context.

% Given any element $g\in\G$,
% we write~$\pi^g$ for the representation $x\mapsto\pi(gxg^{-1})$ of
% $\H^g=g^{-1}\H g$.

\subsection{}
\label{par22}

Let $F$ be a non-archimedean locally compact field of residue 
characteristic~$p$. 
We will write $\oo$ for its ring of integers,~$\pp$ for the maximal ideal 
of $\oo$, 
$\kk$ for its residue field and
$q$ for~the~car\-di\-nality of $\kk$.

Given an algebraically closed field $R$ of characteristic different from $p$
and a square root $q^{1/2}$~of $q$ in $R$,
we write $\nu^{1/2}$ for the unramified $R$-character of $F^\times$
which sends any uniformizer~of
$F$~to $q^{-1/2}$ and $\nu$ for the square of $\nu^{1/2}$,
which is unramified and sends any uniformizer of~$F$~to $q^{-1}$.

\subsection{}
\label{par23}

Let $R$ be an algebraically closed field of characteristic different from $p$,
and $n\>1$ be an~in\-te\-ger.
Given $r$ in\-tegers $n_1,\dots,n_r\>1$ such that $n_1+\dots+n_r=n$ 
and, for each $i=1,\ldots,r$,~a
re\-pre\-sentation $\pi_i$ of $\GL_{n_i}(F)$, 
we write
\begin{equation}
\label{defsupercusp}
\pi_1\times \cdots \times \pi_r
\end{equation}
for the representation of
$\GL_{n}(F)$ obtained by normalized parabolic induction from 
$\pi_1\otimes \cdots \otimes \pi_r$
along the parabolic subgroup generated by upper triangular matrices 
and the standard Levi~sub\-group $\GL_{n_1}(F)\times\dots\times\GL_{n_r}(F)$.

An irreducible representation of $\GL_n(F)$ is said to be
\textit{cuspidal} (respectively, \textit{supercuspidal})~if~it
does not occur~as~a subrepresentation
(respectively, a subquotient)
of a representation of~the form \eqref{defsupercusp} with $r\>2$.
Any supercuspidal representation of $\GL_n(F)$ is cuspidal. 

Given a representation $\pi$ of $\GL_n(F)$ and a character $\chi$ of 
$F^\times$, we will write $\pi\chi$ for $\pi(\chi\circ\det)$.

\subsection{}
\label{defrlintro}

Let us fix
a prime number $\ell\neq p$ and
an algebraic closure $\qlb$ of the field of $\ell$-adic numbers.
Let $\zlb$ denote its ring~of~in\-te\-gers,
and $\flb$ denote the residue field of $\zlb$.

An irreducible $\qlb$-representation~$\pi$ of $\GL_n(F)$
is said to be \emph{integral} if the $\qlb$-vector space~of~$\pi$
contains a stable $\zlb$-lattice.
For any such lattice $L$,
the $\flb$-re\-presentation~$L\otimes\flb$~of $\GL_n(F)$~has
fi\-nite length.
Its semisimplification is independent of the choice of $L$:
it is called~the~\emph{re\-duction~mo\-dulo~$\ell$} of~$\pi$
and is denoted by $\rl(\pi)$.

Given an irreducible $\flb$-representation~$\pi$ of $\GL_n(F)$,
any~in\-tegral irreducible $\qlb$-representation $\widetilde{\pi}$ of $\GL_n(F)$
such that $\r_\ell(\widetilde{\pi})=\pi$ is called a $\qlb$-\emph{lift} of~$\pi$. 

\subsection{}

In the representation theory of the finite general linear group $\GL_n(\k)$, 
one can also define parabolic induction,
cuspidal and supercuspidal representations,
reduction mod $\ell$ and lift,
and~we will use the same notation \eqref{defsupercusp} for parabolic induction
and $\r_\ell$ for reduction mod $\ell$ as in the non-Archimedean case.

\section{The first main result ($\ell\neq2$)}
\label{stamant}

Let $F/F_0$ be a quadratic extension of non-Archimedean locally compact
fields of residue cha\-rac\-te\-ristic $p\neq2$,
$\s$ be its non-trivial automorphism
and $\ell\notin\{2,p\}$ be a prime number. 

Let $G$ be the group $\GL_n(F)$ for some integer $n\>1$,
equipped with the action of $\s$ component\-wise.
Let $H$ denote the closed subgroup $G^{\s}=\GL_n(F_0)$ made of
the $\s$-fixed points of $G$.

% Let $\ell$ be a prime number different from $p$ and $2$.
 
\subsection{}

Let $R$ be an algebraically closed field of characteristic different from $p$.
Given any $R$-repre\-sentation $\pi$ of $\G$,
we write $\pi^\s$ for the representation $g\mapsto\pi(\s(g))$.

We write $\kk_0$ for the residue field of $F_0$ and $q_0$ for its cardinality,
and $\nu_0$ for the unramified~cha\-racter of $F_0^\times$ sending
any uniformizer to $q_0^{-1}$.
Given any square root $q_0^{1/2}$ of $q_0$ in $R$, we set
\begin{equation}
\label{choixq12sec3}
q^{1/2} =
\left\{ 
\begin{array}{ll}
q_0^{1/2} & \text{if $\F/\F_0$ is ramified}, \\ 
q_0 & \text{if $\F/\F_0$ is unramified},
\end{array}\right.
\end{equation} 
which we use to define the character $\nu^{1/2}$ (\S\ref{par22})
and normalized parabolic induction (\S\ref{par23}).

\subsection{}

Here is our first main theorem.

\begin{theo}
\label{CONJNRV}
A cuspidal irreducible $\flb$-representation of $G$ is $H$-distinguished
if and only~if it has~an $H$-distinguished cuspidal lift to $\qlb$.
\end{theo}

Note that:
\begin{itemize}
\item 
any cuspidal $\flb$-representation of $G$ having an 
$H$-distinguished cuspidal lift to $\qlb$~is $H$-dis\-tin\-guished
(\cite{KuMaAsai} Theorem 3.4),
\item
Theorem \ref{CONJNRV} is known to hold for supercuspidal representations
(\cite{VSANT19} Theorem 10.11).
\end{itemize}

To prove our Theorem \ref{CONJNRV},
it thus suffices to prove that any $H$-distinguished
non-su\-per\-cuspi\-dal,
cuspidal $\flb$-representation of $G$
has~an $H$-distinguished cuspidal lift to $\qlb$.

\begin{rema}
\label{soitditenpassant}
If $\ell=2$, there are examples of $\GL_2(F_0)$-distinguished  
non-super\-cuspidal,~cus\-pidal
$\overline{\FF}_2$-repre\-sen\-ta\-tions of $\GL_2(F)$
with no $\GL_2(F_0)$-distinguished lift (\cite{NRV25} Remark 6.4).
\end{rema}

\subsection{}
\label{leuwen3}

Thanks to \cite{NRV25}, 
the proof of Theorem \ref{CONJNRV} can be immediately reduced to the level $0$ 
case. 

\begin{prop}
\label{lamiel}
Suppose that Theorem \ref{CONJNRV}
holds for cuspidal $\flb$-representations of level~$0$ of $\GL_n(F)$,
for all $n\>1$ and all $F/F_0$.
Then it holds for cuspidal representations of any level.
\end{prop}

\begin{proof}
It suffices to prove that any $H$-distinguished cuspidal $\flb$-representation 
of $G$ has~a~$H$-dis\-tin\-guished cuspidal lift to $\qlb$.
Let $\pi$ be a $H$-distinguished cuspidal representation of
$G$.~Asso\-ciated with it in \cite{NRV25} \S4.11,
there are
\begin{itemize}
\item 
a tamely ramified extension $T_0$ of $F_0$ such that 
$T=T_0\otimes_{F_0}F$ is a quadratic extension of $T_0$, 
\item
a positive integer $m$ dividing $n$,
\item
and a cuspidal representation $\pi_{\sf t}$ of 
level $0$ of $\GL_m(T)$. 
\end{itemize}
By \cite{NRV25} Proposition 4.40,
this representation $\pi_{\sf t}$ is $\GL_m(T_0)$-distinguished.
By assumption,
since it has level $0$,
it has a $\GL_m(T_0)$-distinguished cuspidal lift to $\qlb$.
It then follows from \cite{NRV25} Lemma 6.5 that $\pi$ has~an 
$H$-distinguished cuspidal lift to $\qlb$.
\end{proof}

We are thus reduced to proving the following theorem.

\begin{theo}
\label{CONJNRV0}
Any $H$-distinguished non-supercuspidal, cuspidal $\flb$-representation
of level $0$ of $G$~has a $H$-distinguished cuspidal lift to $\qlb$.
\end{theo}

\subsection{}

Let us fix a square root $q_0^{1/2}$ in $\qlb$.
By reduction mod the maximal ideal of $\zlb$,
it defines~a square root in $\flb$,
which we also denote by $q_0^{1/2}$.
By applying the rule \eqref{choixq12sec3},
we get a square root $q^{1/2}$ in $\qlb$ and $\flb$.

Let us recall the classification of non-supercuspidal, cuspidal
$\flb$-representations of $G$ given by \cite{MSc} Théorème~6.14. 
Let $k\>1$ be a positive integer
and $\rho$ be a supercuspidal $\flb$-representation of $\GL_k(\F)$. 
Accor\-ding to \cite{MSc} \S8.1,
for any $r\>1$, the in\-duced representation 
\begin{equation}
\label{INDRHO}
\rho\nu^{-(r-1)/2}\tdt\rho\nu^{(r-1)/2}
\end{equation}
contains a unique generic irreduci\-ble subquotient, denoted $\St_{r}(\rho)$.
(For the definition of a~generic re\-presentation, see \S\ref{generik}.) 
Let $\et(\rho)$ be the smallest~inte\-ger $i\>1$ 
such that $\rho\nu^i$ is~iso\-mor\-phic to~$\rho$.  

\begin{prop}
\label{Coupure}
Let $\pi$ be a non-supercuspidal, 
cuspidal $\flb$-representation of $\GL_n(\F)$. 
\begin{enumerate}
\item 
There are a unique positive integer $r=r(\pi)\>2$ dividing $n$
and a supercuspidal representa\-tion $\rho$ of $\GL_{n/r}(\F)$ 
such that $\pi$ is isomorphic to $\St_r(\rho)$.
\item
There is a unique integer $u\>0$ such that $r=\et(\rho)\ell^u$.
\item
Let $\rho'$ be a supercuspidal representation of $\GL_{n/r}(\F)$.
The representation $\pi$ is isomorphic to $\St_{r}(\rho')$ if and only if 
$\rho'$ is isomorphic to $\rho\nu^i$ for some $i\in\ZZ$.
\end{enumerate} 
\end{prop}

\begin{rema}
\label{BertheJosserand}   
\label{HortenseJosserand}
\begin{enumerate}
\item 
By \cite{MSjl} Lemme 3.6,
the integer $\et(\rho)$ is equal to the order of $q^{t(\rho)}$ mod~$\ell$,
where $t(\rho)$ is the number of unramified characters $\chi$ of $\GL_k(F)$
such that $\rho\chi$ is~iso\-mor\-phic to $\rho$.
\item
The representation $\pi$ has level $0$ if and only if $\rho$ has level $0$,
in which case the number $t(\rho)$ is equal to $k$ (see \cite{MSt} \S3.4). 
\end{enumerate}
\end{rema}

\subsection{}

Recall from \cite{NRV25} the following necessary conditions of distinction.
Let~$e$~and $e_0$ denote the 
orders~of $q$ and $q_0$ mod $\ell$, 
respectively.

\begin{prop}
\label{nessconddist}
Let $\pi$ be an $H$-distinguished non-supercuspidal, cuspidal
representation of le\-vel $0$ of $G$.
Then there is a supercuspidal representation $\rho$ of
$\GL_k(F)$, with $rk=n$, 
such that~$\pi$ is iso\-mor\-phic to $\St_r(\rho)$ and 
\begin{enumerate}
\item 
If $r$ is odd, then:
\begin{enumerate}
\item 
the representation $\rho$ is $\GL_k(F_0)$-distinguished,
\item
if $F/F_0$ is unramified, then $k$, $e$ are odd,
\item
if $F/F_0$ is ramified, then $k$ is even.
\end{enumerate}
\item
If $r$ is even, then: 
\begin{enumerate}  
\item
we have $\rho^{\vee\s}=\rho\nu^{i}$ for some $i\in\{0,1\}$,
\item
if $F/F_0$ is unramified or $k\neq 1$,
then $\rho\nu^{i/2}$ is $\GL_k(F_0)$-distinguished, 
\item
if $F/F_0$ is unramified, then $i=0$.
\end{enumerate}
\end{enumerate}
\end{prop}

\begin{proof}
The case where $r$ is odd comes from \cite{NRV25} Theorem 5.1.
Let us consider the case~where $r$~is even.
First note that $q^{n/2}$ is congruent to $-1$ mod $\ell$ in this case.
In\-deed,
$r=e(\rho)\ell^u$~for~some $u\>0$
and $e(\rho)$ is the order of $q^k$ mod $\ell$ 
(see Remark \ref{HortenseJosserand}).
Now Property (2.a) follows from~\cite{NRV25} Proposition 3.8. 
It follows from~(2.a) that $\mu=\rho\nu^{i/2}$~is
$\s$-self-dual.
It has a unique non-isomorphic $\s$-self-dual unramified twist,
which~is $\mu\nu^{n/2}$.
Suppose that 
$F/F_0$ is unramified or $k\neq 1$.
By \cite{VSANT19} Proposition 10.12,
there~exists exactly~one dis\-tingui\-shed representation among $\mu$ and
$\mu\nu^{n/2}$.
If~it is $\mu\nu^{n/2}$, 
we~may~re\-place~$\rho$ by $\rho\nu^{n/2}$,
which exchanges the roles played by $\mu$ and $\mu\nu^{n/2}$
without changing~$\pi$.
This proves (2.b).

Let us prove (2.c).
Thanks to (2.b),
we may assume that $\rho\nu^{i/2}$ is distinguished.
Thus $\cen_\rho(\w)=q^{ik/2}$.~This implies 
\begin{equation*}
\cen_{\pi}(\w) = \cen_\rho(\w)^r = q^{in/2} = (-1)^i
\end{equation*}
since $q^{n/2} \equiv -1$ mod $\ell$.
Since $\pi$ is distinguished, 
we must have $i=0$.
\end{proof}

\subsection{}
\label{classif6162}

Recall the necessary and sufficient conditions for a 
non-supercuspidal, cuspidal
representa\-tion of level $0$~of~$G$ to have a $H$-distinguished
cuspidal lift (\cite{NRV25} Propositions 6.1,~6.2). 

Let $\cla=\cla_{F/F_0}$ be the character of $F_0^\times$
with~ker\-nel ${\rm N}_{F/F_0}(F^\times)$,
the subgroup of $F/F_0$-norms.

\begin{theo}
\label{DLT1}
A non-supercuspidal, cuspidal representation of level $0$ of $G$
has a $H$-distin\-guished cuspidal lift if and only if it is of the form 
$\St_r(\rho)$ with $\rho$ a supercuspidal representation~of $\GL_k(F)$
with $rk=n$, and one of the following three conditions is satisfied: 
\begin{enumerate}
\item 
$r$ is odd, $F/F_0$ is unramified, $e_0$ is even
and $\rho$ is $\GL_k(F_0)$-distinguished, 
\item
$r$ is odd, $F/F_0$ is ramified, $k$ is even, $n/e$ is odd
and $\rho$ is $\GL_k(F_0)$-distinguished, 
\item
$r$ is even, $F/F_0$ is ramified, $k=1$
and $\rho |_{F_0^\times} \in \{\nu_0^{-1},\cla\}$. 
\end{enumerate}
\end{theo}

It will be convenient to rephrase Theorem \ref{DLT1} as follows.

\begin{theo}
\label{DLT2}
A non-supercuspidal, cuspidal representation of level $0$ of $G$
has a $H$-distin\-guished cuspidal lift if and only if it is of the form 
$\St_r(\rho)$ with $\rho$ a supercuspidal representation~of $\GL_k(F)$
with $rk=n$, and one of the following conditions is satisfied: 
\begin{enumerate}
\item 
$F/F_0$ is unramified, 
$e,k$ are odd, $e_0$ is even and $\rho$ is $\GL_k(F_0)$-distinguished, 
\item
$F/F_0$ is ramified, $e$ is even and
\begin{enumerate}
\item either
$k$ is even with the same dyadic valuation as $e$
and $\rho$ is $\GL_k(F_0)$-distinguished, 
\item
or $k=1$ and $\rho |_{F_0^\times} \in \{\nu_0^{-1},\cla\}$. 
\end{enumerate}
\end{enumerate}
\end{theo}

\begin{proof}
Remark \ref{BertheJosserand} implies that $e(\rho)=e/(e,k)$, thus
\begin{equation}
\label{molde}
r=\ell^ue/(e,k),
\quad
n/e=\ell^uk/(e,k),
\quad
\text{for some $u\>0$.}
\end{equation}
Suppose that $F/F_0$ is unramified. 
Then the fact that $\rho$ is $\GL_k(F_0)$-distinguished~implies that~$k$~is odd
(by \cite{VSANT19} Proposition 9.8).
It then follows from \eqref{molde} that $r$ and $e$ have the same~pa\-ri\-ty.

Suppose now that $F/F_0$ is ramified.
Then $r$ is odd if and only if $v_2(e)\< v_2(k)$,
and $n/e$~is~odd if~and only if $v_2(k)\< v_2(e)$,
thanks to \eqref{molde},
where $v_2(a)$ denotes the dyadic valuation of~a~po\-si\-tive~in\-te\-ger $a\>1$.
\end{proof}

\begin{rema}
When the conditions of Theorem \ref{DLT2} are satisfied,
$e_0$ is even and $e$ has the same parity as the 
ramifica\-tion order of $F/F_0$.
\end{rema}

\begin{rema}
\label{SSDLT}
One can easily extract the following result from Theorem \ref{DLT2} and
the Dicho\-to\-my and Disjunction Property of supercuspidal representations.
We won't use it in the present article, but it is worth mentioning it.
Let $\pi$ be a non-supercus\-pi\-dal, cuspidal represen\-tation~of~le\-vel $0$ 
of $G$, and set $r=r(\pi)$ and $k=n/r$.
Then $\pi$ has a $\s$-self-dual cuspidal lift if and only if~it is~$\s$-self-dual 
and one of the following conditions is satisfied: 
\begin{enumerate}
\item 
$F/F_0$ is unramified, 
$e,k$ are odd and $e_0$ is even, 
\item
$F/F_0$ is ramified, $e$ is even and
either
$k$ has the same dyadic valuation as $e$, 
or $k=1$. 
\end{enumerate}
One easily obtains a similar result for cuspidal represen\-tations of positive level
which uses the~in\-va\-riants $T/T_0$ and $m$ of \cite{NRV25} Propositions 
6.1, 6.2. 
\end{rema}

\subsection{}

As has been pointed out to us by A.~M\'\i nguez,
the conditions of Theorem~\ref{DLT2} only depend on $\rho$
in the following sense:
if $\rho$ is a supercuspidal representation of $\GL_k(F)$
and if we~defi\-ne $\pi_u$ to be the cuspidal representation
$\St_{e(\rho)\ell^u}(\rho)$ for any $u\>0$, then
\begin{itemize}
\item 
  either, for all $u\>0$,
  the representation $\pi_u$ has no distinguished cuspidal lift, 
\item
or, for all $u\>0$,
  the representation $\pi_u$ has a distinguished cuspidal lift.
\end{itemize}
Consequently,
in order to prove Theorem \ref{CONJNRV},
given a $\rho$ such that $\pi_u$ is distinguished for some $u$,
it suffices to find a $v$ such that $\pi_v$ has a distinguished lift.
But we won't use~this~stra\-tegy.

Our strategy is by contradiction:
we assume that there is a cuspidal representation $\pi$
of level $0$ which is distinguished but has no distinguished lift.
We then compute a $\g$-factor in two different manners,
and find two different values
(see Propositions \ref{gammais1} and \ref{sadus}).

\subsection{}

Let us first consider more carefully the case where $r$ is even,
$F/F_0$ is ramified and $k=1$.

\begin{prop}
\label{DLT05}
Suppose that $F/F_0$ is ramified and $n$ is even.
Let $\pi$ be a cuspidal represen\-ta\-tion of level $0$ of $G$ 
such that $r=n$. 
The following conditions are equivalent:
\begin{enumerate}
\item
  $\pi$ is $H$-distinguished.
\item
  $\pi$ has an $H$-distinguished lift.
\item
  $\pi$ is isomorphic to $\St_n(\rho)$ for some 
character $\rho$ of $F^\times$ such that
$\rho |_{F_0^\times} \in \{\nu_0^{-1},\cla\}$. 
\end{enumerate}
\end{prop}

\begin{proof}
We already know that (3) implies (2) and that (2) implies (1).
Let us prove that~(1)~im\-plies (3).
Assume that $\pi$ is distinguished.
By Proposition \ref{nessconddist}(2.a),
it is isomorphic to $\St_n(\rho)$ for some 
character~$\rho$ of $F^\times$ such that $\rho\nu^{i/2}$~is
$\s$-self-dual for some $i\in\{0,1\}$. 
The latter is~thus either distinguished
or $\cla$-distinguished.
Since $F/F_0$ is ramified,
$\cla$~is~ra\-mified,
thus $\rho\nu^{i/2}$ is distinguished if and only if $\rho$ is unramified.
We thus get
\begin{equation}
\label{bathilde}
\rho(\w_0)=
\left\{
\begin{array}{ll}
q^i & \text{if $\rho$ is unramified}, \\ 
\cla(-1)q^i & \text{if $\rho$ is ramified}.
\end{array}
\right.
\end{equation}
Since $\cen_\pi(\w)=\rho(\w)^{n}=\rho(\w_0)^{n/2}$, 
we thus have
\begin{equation}
\label{bathilde2}
\cen_\pi(\w)=
\left\{
\begin{array}{ll}
q^{in/2} & \text{if $\rho$ is unramified}, \\ 
\cla(-1)^{n/2}q^{in/2} & \text{if $\rho$ is ramified}.
\end{array}
\right.
\end{equation}
But, since $F/F_0$ is ramified, 
the fact that $\pi$ is distinguished implies 
(by \cite{NRV25} Theorem 4.45,~Lem\-ma 6.12) 
that
\begin{equation}
\label{bathilde3}
\cen_\pi(\w)=
\left\{
\begin{array}{ll}
-1 & \text{if $\rho$ is unramified}, \\ 
\cla(-1)^{n/2} & \text{if $\rho$ is ramified}.
\end{array}
\right.
\end{equation}
We thus get 
\begin{itemize}
\item 
if $\rho$ is unramified, then $i=1$, thus $\rho |_{F_0^\times} = \nu_0^{-1}$,
\item
if $\rho$ is ramified, then $i=0$, thus $\rho |_{F_0^\times} = \cla$.
\end{itemize}
This proves the proposition.
\end{proof}

\subsection{}
\label{mammuth}

It follows that,
if $\pi$ is a non-supercuspidal, cuspidal representation of level $0$ of $G$
which is $H$-distinguished but has no $H$-distinguished cuspidal lift,
then $\pi$ is isomorphic to $\St_r(\rho)$ and
we are~in one~of the four following cases:
\begin{enumerate}
\item 
$F/F_0$ is unramified, $e_0$ is odd and $\rho$ is 
$\GL_k(F_0)$-distinguished, 
\item 
$F/F_0$ is unramified, $e$ is even and $\rho$ is 
$\GL_k(F_0)$-distinguished,
\item
$F/F_0$ is ramified, $k/(k,e)$ is even, $e/(e,k)$ is odd and
$\rho$ is $\GL_k(F_0)$-distinguished, 
\item
$F/F_0$ is ramified, $k$ is even and $\rho\nu^{i/2}$ is
$\GL_k(F_0)$-distinguished for an $i\in\{0,1\}$.
\end{enumerate} 

\subsection{}

Let $\psi$ be a non-trivial character of $F$. 
Associated with any generic irreducible represen\-tation $\tau$ of $\GL_{m}(F)$
with $m\<n$, there is a Rankin--Selberg $\g$-factor
$\g(X,\pi,\tau,\psi) \in\flb(X)$ whose definition and main properties are
recalled in Section \ref{BGRSgamma}.

\begin{lemm}
\label{CalculMagique}
Let $\pi$ denote a non-supercuspidal, cuspidal $\flb$-representation of $G$
isomorphic to $\St_r(\rho)$ for some supercuspidal representation
$\rho$ of $\GL_k(F)$. 
Then
\begin{equation*}
\g \left( q^{-1/2},\pi,\rho^\vee,\psi \right) = 
\cen_{\pi}(-1)^{k-1} \cdot (-1)^{r} \cdot q^{n/2} 
\end{equation*} 
where $\cen_\pi$ is the central character of $\pi$.
\end{lemm}

\begin{rema}
Note that this value is a sign, since $n$ is a multiple of the order $e$
of $q$ mod $\ell$.
\end{rema}

The proof of this lemma is postponed to \S\ref{PreuveCalculMagique}.
We deduce from it the following result.

\begin{prop}
\label{gammais1}
Let $\pi$ be a non-supercuspidal, cuspidal representation of level 
$0$ of $G$,~iso\-mor\-phic to $\St_r(\rho)$ for some supercuspidal
representation $\rho$ of $\GL_k(F)$.
Suppose that $\pi$ is $H$-dis\-tinguished but has no
$\H$-distinguished lift to $\qlb$.
Then
\begin{equation*}
\g(q^{-1/2},\pi,\rho^\vee,\psi) = -1.
\end{equation*} 
\end{prop}

\begin{proof}
According to Lemma \ref{CalculMagique}, 
and since $\pi$ is distinguished,
we have
\begin{equation*}
\g \left( q^{-1/2},\pi,\rho^\vee,\psi \right) = 
(-1)^{r} \cdot q^{n/2}. 
\end{equation*} 
Suppose first that $r$ is even.
As in the proof of Proposition \ref{nessconddist}, 
this implies that $q^{n/2}$ is congruent to $-1$ mod $\ell$.
The result follows. 

Now suppose that $r$ is odd.
We are thus either in Case 1 or in Case 3 of \S\ref{mammuth}. 
If we are~in~Ca\-se~1,
then $q^{1/2}=q_0$ since $F/F_0$ is unramified (see \eqref{choixq12sec3})
and $e_0=e$ since $e_0$ is odd. 
We thus have $q^{n/2} = q_0^n = 1$ since $e$ divides $n$.
In Case 3,
then $q^{n/2}=(q^e)^{k/2(k,e)}=1$ since $k/(k,e)$ is even.
\end{proof}

In order to prove Theorem \ref{CONJNRV0},
it thus suffices to prove:

\begin{prop}
\label{sadus}
Let $\pi$ be a non-supercuspidal cuspidal representation of level 
$0$ of $G$,~iso\-mor\-phic to $\St_r(\rho)$ for some supercuspidal
representation $\rho$ of $\GL_k(F)$.
Suppose that $\pi$ is $H$-dis\-tin\-guished but has no
$\H$-distinguished lift to $\qlb$. 
Then 
\begin{equation*}
\g(q^{-1/2},\pi,\rho^\vee,\psi) = 1. 
\end{equation*}
\end{prop}

Since $\ell\neq2$,
this will give us the expected contradiction.

We are now going to reduce~the~proof of Proposition \ref{sadus}
to that of a similar proposition~for re\-presen\-tations of general linear
groups~over the finite field $\k$.

\subsection{}
\label{typelevel0}

Let $V$ denote the vector space of $\pi$,
and consider the subspace $W$ of $V$ made of the~vectors 
invariant by the subgroup $1+\Mat_n(\pp)$.
Since $\pi$ has level $0$,
this subspace is non-zero.
It carries~an action of $\GL_n(\oo)$ which factors
into a representation $\overline{\pi}$ of $\GL_n(\k)$ on $W$ 
which is irreducible and cuspidal
(\cite{Vigbook} III.3.14 or \cite{MSc} Exemple~5.10,
Proposition 5.11).
Similarly,
one defines a cuspidal, irreducible representation $\overline{\rho}$
of $\GL_k(\k)$ on the vectors of the space of $\rho$ invariant by
$1+\Mat_k(\pp)$.

By \cite{NRV25} 4.7,
the representation $\overline{\rho}$ is supercuspidal
and $\overline{\pi}$ occurs as a subquotient
of the induced representation 
$\overline{\rho}^{\times r}=\overline{\rho}\times\dots\times\overline{\rho}$ 
(where $\overline{\rho}$ occurs $r$ times).
On the other hand,
it follows from~\cite{NRV25} Theorem 4.45 that
$\overline{\pi}$ is distinguished by the subgroup
$\overline{H}$ of $\GL_n(\k)$ defined by
\begin{equation*}
\overline{H} = 
\left\{ 
\begin{array}{ll}
\GL_n(\k_0) & \text{if $\F/\F_0$ is unramified}, \\ 
\GL_{n/2}(\k)\times\GL_{n/2}(\k) & \text{if $\F/\F_0$ is ramified},
\end{array}\right.
\end{equation*} 
and from \cite{NRV25}~Lemma 6.5 
that it has no $\overline{H}$-distinguished lift. 
Note that $n$ even when $\F/\F_0$~is~ra\-mified,
thanks to Proposition \ref{nessconddist}.

Suppose that~$\psi$~has conductor $1$,
that is,
$\psi$ is trivial on $\pp$ but not on $\oo$.
It defines a non-tri\-vial character of $\k$,
which we still denote by $\psi$.
We then may associate to the pair $(\overline{\pi},\overline{\rho}^\vee)$ a 
non-zero scalar $\g(\overline{\pi},\overline{\rho}^\vee,\psi) \in R^\times$ 
which will be defined in Section \ref{duveyrier}.

We have the following proposition,
the proof of which is postponed to \S\ref{Preuvepadicfinite}.

\begin{prop} 
\label{padicfinite}
Suppose that~$\psi$~has conductor $1$. 
Then
\begin{equation*}
\g(q^{-1/2},\pi,\rho^\vee,\psi) = \g(\overline{\pi},\overline{\rho}^\vee,\psi).
\end{equation*} 
\end{prop}

The proof of Proposition \ref{sadus}
thus reduces to proving that
$\g(\overline{\pi},\overline{\rho}^\vee,\psi) = 1$. 

\subsection{}
\label{parlev}

More generally,
for any non-supercuspidal,
cuspidal representation $\varkappa$ of $\GL_n(\k)$,~there~are
a uni\-que integer $u$ dividing $n$
and a unique supercuspidal representation $\varrho$ of
$\GL_{n/u}(\k)$~such~that~$\varkappa$ occurs as a subquotient
of the parabolically induced representation $\varrho^{\times u}$
(see \cite{NRV25} Proposi\-tion~3.9 for instance).
We will say that $\varrho$ is the supercuspidal representation
\textit{associated with}~$\varkappa$.

Changing the notation,
the proof of Theorem \ref{CONJNRV} is thus reduced to the proof of the 
following proposition,
which entirely pertains to the theory of representations of $\GL_n(\k)$.

\begin{prop}
\label{minelli}
Let $\pi$ be a non-supercuspidal cuspidal representation of
$\GL_n(\k)$,
with~as\-so\-ciated supercuspidal representation $\rho$.
Let $H$ be one of the following subgroups of $\GL_n(\k)$:
\begin{enumerate}
\item 
either $\k$ has a subfield $\k_0$ such that $[\k:\k_0]=2$ 
and $H=\GL_n(\k_0)$,
\item
or $n=2m$ for some integer $m\>1$ and $H$ is the standard
Levi subgroup $\GL_{m}(\k)\times\GL_{m}(\k)$.
\end{enumerate}
Suppose that $\pi$~is ${H}$-distingui\-shed
{but has no ${H}$-distinguished lift.}
Then,
for~any~non-trivial~cha\-rac\-ter $\psi$ of $\k$,
we have 
$\g(\pi,\rho^\vee,\psi) = 1$.
\end{prop}

We will refer to the first case as the Galois case,
and to the second case as the Levi case.

In conclusion,
in order to prove our main Theorem \ref{CONJNRV},
it remains to prove Lemma \ref{CalculMagique},~Pro\-po\-sition 
\ref{padicfinite} and Proposition \ref{minelli}.

\subsection{}

Before proceeding to the proof of these results, 
let us end this section by the following~Dis\-junction Theorem,
well-known for discrete series representations of $G$ when $R=\CC$.

\begin{coro}
\label{disjunctionmodl}
A $\s$-self-dual cuspidal $\flb$-representation of $\GL_n(\F)$
cannot be both dis\-tin\-gui\-shed and $\cla$-distinguished.
\end{coro}

\begin{proof}
If $r$ is odd, 
this is \cite{NRV25} Corollary 5.2.
Assume that $r$ is even and $\pi$ is both distinguished
and $\cla$-distinguished.
We may and will~assu\-me that $\pi$ has level $0$.
By Theorem \ref{CONJNRV} together with Theorem \ref{DLT1}, 
$\F/\F_0$ is~ra\-mi\-fied, $r=n$ and 
$\pi$ is isomorphic to $\St_r(\rho)$ for some
character $\rho$ of $F^\times$ such that
$\rho|_{F_0^\times}\in\{\nu_0^{-1},\cla\}$.

Let $\a$ be a tamely ramified cha\-rac\-ter of $\F^\times$
extending $\cla$.
Since $\pi$ is $\cla$-distinguished, 
the~repre\-sentation $\pi\a=\St_r(\rho\a)$ is distinguished.
It is thus isomorphic to $\St_r(\rho')$
for~some character $\rho'$ of $F^\times$~such that
$\rho'|_{F_0^\times}\in\{\nu_0^{-1},\cla\}$.
Since $\pi$ is  isomorphic to $\St_r(\rho'\a^{-1})$,
it follows from the~classifica\-tion of~cus\-pidal~repre\-sen\-tations
that $\rho'\a^{-1}=\rho\nu^{i}$ for some $i\in\ZZ$.

Since $\cla$ is ramified and $\nu|_{F_0^\times}=\nu_0^2$,
we have a contradiction in any case.
\end{proof}

\begin{rema}
\label{haydee}
Note that Dichotomy does not hold:
when $r$ is even and either $\F/\F_0$ is unra\-mi\-fied,
or $\F/\F_0$ is ramified and $r\neq n$, 
any $\s$-self-dual cuspidal representation of level $0$
is neither distingui\-shed nor $\cla$-distinguished.
\end{rema}

\section{Rankin--Selberg factors over non-Archimedean fields}
\label{BGRSgamma}

In this section, $\ell$~is any pri\-me number different from $p$
and $\R$ is an algebraically closed field~of characteristic~$0$~or $\ell$.
We fix a non-trivial character
\begin{equation*}
\label{psiFd}
\psi : \F \to \R^\times
\end{equation*}
with conductor $d=d(\psi)$,
which is the smallest integer $i\in\ZZ$ such that $\pp^{i} \subseteq \Ker(\psi)$. 
The purpose of this section is to~:
\begin{itemize}
\item 
recall the definition and main properties of the Rankin--Selberg local factors 
of \cite{KM17}, 
\item
prove Lemma \ref{CalculMagique}. 
\end{itemize}

\subsection{} 
\label{generik}

Given an integer $n\>1$, 
we write $\G=\GL_n(\F)$. 
Let $\N=\N_n$ be the subgroup of $\G$~made~of all~up\-per triangular
unipotent matrices. 
The cha\-rac\-ter
\begin{equation*}
x\mapsto\psi(x_{1,2}+\dots+x_{n-1,n})
\end{equation*}
of $\N$ will still be denoted by $\psi$. 
An irreducible $\R$-representation $\pi$ of $\G$ is \textit{generic} if
it embeds in the space $\Ind^{\G}_{\N}(\psi)$ made of all smooth functions
$\W:\G\to\R$ such that $\W(xg)=\psi(x)\W(g)$~for all~ $x\in\N$
and~$g\in\G$,~or equivalently,
if the $\R$-vector space $\Hom_{\N}(\pi,\psi)$
is non-zero.  
When this is the case,
the latter has di\-men\-sion $1$,
that is,
the image of the embedding of $\pi$ in
$\Ind^{\G}_{\N}(\psi)$~is~uni\-que.~It
is called the \textit{Whittaker model} of $\pi$ with respect
to $\psi$ and is denoted $\WW(\pi,\psi)$.

\subsection{} 

Given integers $n$ and $m$ such 
that~$1\<m\<n$,
let us review the Rankin--Selberg functional equa\-tions and local factors 
for~$\GL_n(\F)\times\GL_m(\F)$ following~\cite{KM17}.
Let us fix a square root
\begin{equation}
\label{choixq12}
q^{1/2} \in \R^\times
\end{equation}
of $q$.
Let $\pi$ and $\pi'$ be generic $\R$-representations 
of $\GL_n(\F)$ and $\GL_m(\F)$, respecti\-ve\-ly. 
Attached to the pair $(\pi,\pi')$ in \cite{KM17}~Section 3, there are
\begin{enumerate}
\item 
an \emph{$\L$-factor} $\L(\X,\pi,\pi')$,
which is an element of $\R(\X)$
of the form $\P^{-1}$ where~$\P$
is a polyno\-mial in $\R[\X]$
such that $\P(0)=1$,
\item
an \emph{$\e$-factor} $\e(\X,\pi,\pi',\psi)$,
which is of the form
\begin{equation}
\label{epsmonomialR}
\e(\X,\pi,\pi',\psi) = e \X^{f(\pi,\pi',\psi)}
\end{equation}
for a non-zero scalar $e\in\R^\times$
and an integer $f(\pi,\pi',\psi)\in\ZZ$, 
\item
and a \emph{$\g$-factor} $\g(\X,\pi,\pi',\psi)$,
which is the element of $\R(\X)$ de\-fi\-ned by 
\begin{equation*}
\label{eq:2}
\g(\X,\pi,\pi',\psi) =
\e(\X,\pi,\pi',\psi) \cdot \L(\X,\pi,\pi')^{-1} \cdot
\L\left(\frac 1 {q\X},{\pi}^\vee,{\pi}'^\vee\right).
\end{equation*}
\end{enumerate}
We give more details,
for which we refer to \cite{KM17} Theorem~3.5, Co\-rol\-la\-ry 3.11.

\subsection{} 
\label{discdepq12}

Assume first that $m=n$.
Then the $\L$-factor
$\L(\X,\pi,\pi')$~is the unique Euler~fac\-tor genera\-ting
the~frac\-tio\-nal ideal of $\R[\X^{\pm1}]$~ge\-ne\-ra\-ted by the Laurent series
\begin{eqnarray*}
\I(\X,\W,\W',\Phi) &=& \sum\limits_{k\in\ZZ} \ \X^{k} \cdot 
\int\limits_{\boldsymbol{{\sf Y}}_{n,k}}
\W(g)\W'(g)\Phi( \row g)\ {\rm d}g, \\
\notag& & \W\in\WW(\pi,\psi), \quad \W'\in\WW(\pi',\psi^{-1}),
\quad \Phi\in\Cc_{\rm c}^\infty(\F^n),
\end{eqnarray*}
where
$\Cc_{\rm c}^\infty(\F^n)$ is the space of locally constant,
compactly supported functions from $F^n$ to $R$,~the
integral is over the set $\boldsymbol{{\sf Y}}_{n,k}$
made of all $\N_{n}g \subseteq \GL_n(\F)$
such that~$\det(g)$~has~va\-luation~$k$ and
$\row$~is the row matrix $(0 \ \dots \ 0\ 1)$.
We have the functional equation 
\begin{equation} 
\label{FEpadic}
\I\left(\frac {1} {q\X},\widetilde{\W},\widetilde{\W}',\widehat{\Phi}\right)
= \cen_{\pi'}(-1)^{n-1}\cdot \g(\X,\pi,\pi',\psi) \cdot \I(\X,{\W},{\W}',{\Phi})
\end{equation} 
where $\widetilde{\W}\in\WW(\pi^\vee,\psi^{-1})$
is the function given by $g\mapsto\W(\wn g^{*})$, 
where $\wn$ is the antidiagonal~per\-mu\-ta\-tion ma\-trix of maximal length
and $g^*$ is the transpose of $g^{-1}$,
and
\begin{equation*}
\widehat{\Phi} : y \mapsto \int\limits_{\F^n} \Phi(x)\psi(x\cdot y)\ {\rm d}x
\end{equation*}
is the Fourier transform of $\Phi$ with respect to the unique Haar measure 
${\rm d}x={\rm d}x(\psi)$
on $\F^n$ giving measure $q^{nd/2}$ to the lattice $\oo^n$
(and $x\cdot y$ denotes the canonical scalar product of $x,y\in\F^n$).

Assume now that $m<n$
and fix an integer $j\in\{0,\dots,n-m-1\}$.
The $\L$-factor $\L(\X,\pi,\pi')$ is the unique Euler factor generating
the fractional ideal of $\R[\X^{\pm1}]$~ge\-ne\-ra\-ted by the series
\begin{eqnarray}
\label{babilou}
\I_j(\X,\W,\W') &=& \sum\limits_{k\in\ZZ} \ (q^{(n-m)/2}\X)^{k} \cdot 
\int\limits_{\Mat_{j,m}(\F)}
\int\limits_{\boldsymbol{{\sf Y}}_{m,k}}
\W
\begin{pmatrix} g & 0 & 0 \\ x & 1_{j} & 0 \\ 0 & 0 & 1_{n-m-j} \end{pmatrix}
\W'(g)\ {\rm d}g\ {\rm d}x, \\
\notag& & \W\in\WW(\pi,\psi), \quad \W'\in\WW(\pi,\psi^{-1}),
\end{eqnarray}
where the integral over the vector space $\Mat_{j,m}(\F)$
is with respect to the unique Haar measure~gi\-ving 
measure $q^{jmd/2}$ to the~lat\-tice $\Mat_{j,m}(\oo)$.
We have the functional equation 
\begin{equation*}
\I_{n-m-1-j}\left(\frac {1} {q\X},w_{m,n-m}\cdot\widetilde{\W},\widetilde{\W}'\right)
= \cen_{\pi'}(-1)^{n-1}\cdot \g(\X,\pi,\pi',\psi) \cdot 
\I_j(\X,{\W},{\W}')
\end{equation*} 
where $w_{m,n-m}={\rm diag}({\rm id}_m,w_{n-m})$ in $\GL_{n}(\F)$.

\subsection{} 

These Rankin--Selberg local factors enjoy the following properties. 

\begin{prop}
\label{gammaintR}
Let $\pi$ and $\pi'$ be as above. 
\begin{enumerate}
\item 
The integer
\begin{equation*}
f(\pi,\pi')=f(\pi,\pi',\psi) + mnd(\psi)
\end{equation*}
does not depend on $\psi$.
\item
Given an $a\in\F^\times$, let $\psi^a$ denote the character
$x\mapsto\psi(ax)$ of $\F$. 
Then
\begin{equation*}
\label{epspsia}
\e(\X,\pi,\pi',\psi^a)=\cen_\pi(a)^m \cdot \cen_{\pi'}(a)^n
\cdot \left(\qe^{1/2}\X\right)^{mn\cdot{\rm val}_\F(a)}
\cdot \e(\X,\pi,\pi',\psi).
\end{equation*} 
\item 
If $\chi$ is an unramified character of $\F^\times$,
then
\begin{equation*}
\label{eps1nr}
\e(\X,\pi\chi,\pi',\psi) =
\e(\X,\pi,\pi'\chi,\psi) =
\chi(\w)^{f(\pi,\pi',\psi)} \cdot \e(\X,\pi,\pi',\psi)
\end{equation*}
for any uniformizer $\w$ of $\F$. 
\item
Let 
$\L^*(\X,\pi,\pi')$ and 
$\e^*(\X,\pi,\pi',\psi)$
be the local constants obtained by replacing \eqref{choixq12}
by~the opposite square root $-q^{1/2}$.
Then
\begin{eqnarray*}
\L^*(\X,\pi,\pi') &=& \L((-1)^{n-m}\X,\pi,\pi'), \\
\e^*(\X,\pi,\pi',\psi) &=&
(-1)^{(n-m)f(\pi,\pi') + mnd(\psi)}\cdot\e(\X,\pi,\pi',\psi),
\end{eqnarray*}
and the integer $f(\pi,\pi')$ does not depend on the choice
of the square root made in \eqref{choixq12}.
\end{enumerate}
\end{prop}

\begin{proof}
The first three properties are well-known when $\R$ is the field of complex
numbers.~For a general $\R$,
the first property follows from the second one together with the fact that
$d(\psi^a)=d(\psi)-{\rm val}_\F(a)$,
and the third one easily follows from the functional equations.
We now prove~the second and fourth ones,
treating the cases $n=m$ and $m<n$ separately.

Assume first that $n=m$.
If one replaces the square root $q^{1/2}$ by its opposite $-q^{1/2}$,
the~series $\I(\X,\W,\W')$ remain unchanged,
which gives $\L^*(\X,\pi,\pi')=\L(\X,\pi,\pi')$,
and the Haar measure~${\rm d}x$ on $\F^n$
is changed to $(-1)^{nd}\cdot{\rm d}x$,
which implies that
\begin{equation*}
\e^*(\X,\pi,\pi',\psi)=(-1)^{nd}\cdot\e(\X,\pi,\pi',\psi)
\end{equation*}
as expected.
Given $a\in\F^\times$, we define $t=t(a)\in\G$ to be the diagonal matrix
${\rm diag}(a^{n-1},\dots,a,1)$.
The functions $\W^a:g\mapsto\W(tg)$ and $\W'^a:g\mapsto\W'(tg)$ 
belong to $\WW(\pi,\psi^a)$ and $\WW(\pi',\psi^a)$,~res\-pectively.
We have
\begin{eqnarray*}
\I(\X,\W^a,\W'^a,\Phi) &=& \sum\limits_{k\in\ZZ} \ \X^{k} \cdot 
\int\limits_{\boldsymbol{{\sf Y}}_{n,k}}
\W(tg)\W'(tg)\Phi( \row g)\ {\rm d}g \\
& = & \sum\limits_{k\in\ZZ} \ \X^{k} \cdot 
\int\limits_{\boldsymbol{{\sf Y}}_{n,k+{\rm val}_\F(a)\cdot n(n-1)/2}}
\W(h)\W'(h)\Phi( \row h)\ {\rm d}(t^{-1}h) \\ 
& = & \X^{-{\rm val}_\F(a)\cdot n(n-1)/2} \cdot \mu(a) \cdot
\I(\X,\W,\W',\Phi) 
\end{eqnarray*}
where $\mu$ is the character of $\F^\times$ such that
${\rm d}(t^{-1}h) = \mu(a) \cdot {\rm d}h$ on $N_n \backslash \GL_n(F)$.
For all $g\in\G$,~one has
$\W(t\ww g^*)=\widetilde{\W}(a^{1-n}tg)$.
If we denote by $\widehat{\Phi}^{a}$ the Fourier transform of $\Phi$
with respect to the measure 
${\rm d}x(\psi^a)=|a|^{-n}\cdot{\rm d}x(\psi)$,
we have $\widehat{\Phi}^{a}(x)=|a|^{n/2}\cdot\widehat{\Phi}(ax)$
for all $x\in\F^n$.
This gives
\begin{eqnarray*}
\I\left(\frac {1} {q\X},\widetilde{\W^a},\widetilde{\W'^a},\widehat{\Phi}^{a}\right)
&=& \sum\limits_{k\in\ZZ} \ q^{-k}\X^{-k} \cdot 
\int\limits_{\boldsymbol{{\sf Y}}_{n,k}}
\widetilde{\W}(a^{1-n}tg)\widetilde{\W}'(a^{1-n}tg)
\widehat{\Phi}^{a}( \row g)\ {\rm d}g \\ 
& = & \cen_\pi(a)^{n-1} \cdot \cen_{\pi'}(a)^{n-1} \cdot % |a|^{n/2} \cdot
\sum\limits_{k\in\ZZ} \ q^{-k}\X^{-k} \cdot 
\int\limits_{\boldsymbol{{\sf Y}}_{n,k}}
\widetilde{\W}(tg)\widetilde{\W}'(tg)
\widehat{\Phi}^{a}(\row g)\ {\rm d}g \\
& = & \cen_\pi(a)^{n-1} \cdot \cen_{\pi'}(a)^{n-1} \cdot % |a|^{n/2} \cdot
\X^{{\rm val}_\F(a)\cdot n(n-1)/2} \cdot \mu(a) \cdot
\I\left(\frac                                                              {1}
      {q\X},\widetilde{\W},\widetilde{\W}',\widehat{\Phi}^{a}\right) 
\end{eqnarray*}
and
\begin{eqnarray*}
\I\left(\frac {1} {q\X},\widetilde{\W},\widetilde{\W}',\widehat{\Phi}^{a}\right) 
&=& |a|^{n/2} \cdot \sum\limits_{k\in\ZZ} \ q^{-k}\X^{-k} \cdot 
\int\limits_{\boldsymbol{{\sf Y}}_{n,k}}
\widetilde{\W}(g)\widetilde{\W}'(g)
\widehat{\Phi}(a \row g)\ {\rm d}g \\
&=& |a|^{n/2} \cdot \cen_\pi(a) \cdot \cen_{\pi'}(a) \cdot J_{k,a}(X)
\end{eqnarray*}
with 
\begin{eqnarray*}
J_{k,a}(X) &=& \sum\limits_{k\in\ZZ} \ q^{-k}\X^{-k} 
               \int\limits_{\boldsymbol{{\sf Y}}_{n,k+{\rm val}_F(a)\cdot n}} 
\widetilde{\W}(g)\widetilde{\W}'(g)
\widehat{\Phi}( \row g)\ {\rm d}(a^{-1}g) \\
&=& (q\X)^{{\rm val}_\F(a)\cdot n} \cdot |a|^{n(n+1)/2} \cdot 
\I\left(\frac {1} {q\X},\widetilde{\W},\widetilde{\W}',\widehat{\Phi}\right),
\end{eqnarray*}
which implies the expected result.

Assume now that $m<n$,
and write $\I_j^*(\X,\W,\W')$ for the series \eqref{babilou}
defined by replacing~\eqref{choixq12}
by the opposite square root $-q^{1/2}$.
Taking into account that the measure ${\rm d}x$
on $\Mat_{j,m}(\F)$ changes, we get
\begin{equation*}
\I_j^*(\X,\W,\W') = (-1)^{jmd(\psi)}\cdot\I_j((-1)^{n-m}\X,\W,\W'), 
\end{equation*}
thus $\L^*(\X,\pi,\pi')=\L((-1)^{n-m}\X,\pi,\pi')$ and
\begin{eqnarray*}
\e^*(\X,\pi,\pi',\psi) &=&
(-1)^{jmd(\psi)+(n-m-1-j)md(\psi)}\cdot\e((-1)^{n-m}\X,\pi,\pi',\psi) \\
&=& (-1)^{(n-m-1)md(\psi) + (n-m)f(\pi,\pi',\psi)}\cdot\e(\X,\pi,\pi',\psi) \\
&=& (-1)^{(n-m)f(\pi,\pi') + (n(n-m)+n-m-1)md(\psi)}\cdot\e(\X,\pi,\pi',\psi) \\
&=& (-1)^{(n-m)f(\pi,\pi') + mnd(\psi)}\cdot\e(\X,\pi,\pi',\psi) 
\end{eqnarray*}
where the second identity uses \eqref{epsmonomialR},
the third one uses (1) and the last one uses that
the integers $mn$ and $m(n(n-m)+n-m-1)$ have the same parity.
The proof of (2) when $m<n$ follows the same pattern as in the case when 
$m=n$. 
\end{proof}

\subsection{}

A fundamental property that we will use is that the $\g$-factor is 
multiplicative:
if the~generic representations $\pi$ and $\pi'$ are subquotients of
$\rho_1\times \dots \times \rho_r$ and $\mu_1\times \dots \times \mu_s$,
respectively, 
where the $\rho_i$ and the $\mu_j$ are generic, 
then
\begin{equation}
\label{gammaprod}
\g(\X,\pi,\pi',\psi) = \prod_{i=1}^{r}\prod_{j=1}^{s} \g(\X,\rho_i,\mu_j,\psi)
\end{equation}
(see \cite{KM17} Theorem 4.1). 

\subsection{}

Any non-trivial $\qlb$-character of $\F$ takes values in $\overline{\ZZ}_\ell$.
It thus may~be reduced~mod~$\ell$~and~its reduction mod $\ell$ is a 
non-trivial $\flb$-character of $\F$. 
Moreover,
reduction mod $\ell$ induces a bijection between 
non-trivial $\qlb$-characters of $\F$ and
non-trivial $\flb$-character of $\F$.

In this paragraph, 
we fix a non-trivial $\qlb$-character $\psi$ of $\F$
and denote its reduction mod $\ell$~by $\psi$ as well.
The~following~lem\-ma~is \cite{KM17} Lemma 2.27. 

\begin{lemm}
Let $\pi$ be an integral generic $\qlb$-representation of~$\G$.
\begin{enumerate}
\item 
There is,
up to isomorphism,
a unique generic $\qlb$-representation of~$\G$ 
occurring in the~re\-duc\-tion mod $\ell$ of $\pi$.
\item
It occurs with multiplicity $1$ as an irreducible component of
the~re\-duc\-tion mod $\ell$ of $\pi$.
\end{enumerate} 
\end{lemm}

We fix a square root of $q$ in $\qlb$,
which we will use to define the local constants~of
pairs~of~gene\-ric~$\qlb$-representations.
For local constants of~pairs of generic $\flb$-representations,~we
use the~reduc\-tion mod $\ell$ of this square root.

\begin{prop}
\label{gammaint}
Let $\widetilde{\pi}$, $\widetilde{\pi}'$ be integral generic 
$\qlb$-representations of $\GL_n(\F)$, $\GL_m(\F)$, res\-pec\-tively,
and $\pi$, $\pi'$ be the unique generic subquotients
of their reduction mod $\ell$, respectively.
\begin{enumerate}
\item 
We have
$\L(\X,\widetilde{\pi},\widetilde{\pi}')^{-1} \in 1 + \X\zlb[\X]$ and
$\e(1,\widetilde{\pi},\widetilde{\pi}',\psi)\in\overline{\ZZ}{}_{\ell}^\times$. 
\item
The reduction mod $\ell$ of $\g(\X,\widetilde{\pi},\widetilde{\pi}',\psi)$
is equal to $\g(\X,\pi,\pi',\psi)$. 
\end{enumerate}
\end{prop}

\begin{proof} 
The first part follows from \cite{KM17} Corollary 3.6
(in the statement of which the assumption that~$\widetilde{\pi}$
and $\widetilde{\pi}'$ are integral 
is missing) 
and \cite{KM17} Lemma 3.12.
The second part follows from \cite{KM17} Theorem 3.13.
\end{proof}

\subsection{}
\label{PreuveCalculMagique}

Let us prove  Lemma \ref{CalculMagique}. 
Recall that $\pi$ is a non-supercuspidal, cuspidal $\flb$-representation of
$G=\GL_n(F)$ isomorphic to~$\St_r(\rho)$ for some supercuspidal
representation $\rho$ of $\GL_k(F)$. 

\begin{proof}[Proof of Lemma \ref{CalculMagique}]
Let $\widetilde{\rho}$ be a $\qlb$-lift of $\rho$,
whose existence is granted by \cite{Vigbook} III.5.10.
By Proposition \ref{gammaint}(2), we have
\begin{equation*}
\g \left( \X,\pi,\rho^\vee,\psi \right)
= \r_\ell\left( \g(\X,\St_r(\rt),\rt^\vee,\psi) \right).
\end{equation*}
By multiplicativity \eqref{gammaprod}, we have 
\begin{equation*}
  \g \left( \X,\St_r(\rt),\rt^\vee ,\psi \right)
  = \prod\limits_{i=0}^{r-1} \g \left( \X,\rt\nu^{i-(r-1)/2},\rt^\vee ,\psi \right).
\end{equation*}
Write $a=(r-1)/2$ for simplicity. 
Then 
\begin{equation*}
\g \left( \X,\rt\nu^{i-(r-1)/2},\rt^\vee,\psi \right) 
= \e \left( \X,\rt\nu^{i-a},\rt^\vee,\psi \right)
\cdot \frac
{\L( q^{-1} \X^{-1},\rt^\vee\nu^{a-i},\rt)}
{\L(\X,\rt\nu^{i-a},\rt^\vee)}.
\end{equation*}
By \cite{KM17} Theorem 4.9, we have
\begin{equation*}
\L(\X,\rt\nu^{i-a},\rt^\vee)
= \frac {1} {1- q^{k(a-i)}\X^k}
\quad\text{and}\quad
\L(\X,\rt^\vee\nu^{a-i},\rt)
= \frac {1} {1-q^{k(i-a)}\X^k}.
\end{equation*} 
We thus have
\begin{equation}
\label{alanbeck}
\g(\X,\St_r(\rt),\rt^\vee,\psi)
= \e(\X,\rt,\rt^\vee,\psi)^{r} \cdot \prod\limits_{i=0}^{r-1}
\frac { 1 -  q^{k(a-i)}X^k } { 1 - q^{k(i-a-1)}X^{-k} }.
\end{equation}
Let us first focus on the product on the right hand side of \eqref{alanbeck}. 
We have
\begin{eqnarray*}
\prod\limits_{i=0}^{r-1}
\frac { 1 -  q^{k(a-i)}X^k } { 1 - q^{k(i-a-1)}X^{-k} } 
&=& (-X^k)^{r} \cdot \prod\limits_{i=0}^{r-1}
\frac { 1 - q^{k(a-i)}X^k } { q^{k(i-a-1)} - X^k } \\
&=& 
(-1)^r \cdot X^{n} \cdot q^n \cdot 
\prod\limits_{i=0}^{r-1} \frac { 1 - q^{k(a-i)}X^k } { 1 - q^{k(a-i+1)}X^k } \\
&=& 
(-1)^r \cdot (qX)^n \cdot
\frac { 1 - q^{-ka}X^k } { 1 - q^{k(a+1)}X^k }
\end{eqnarray*}
which we may write under the form 
\begin{equation}
\label{C1}
\prod\limits_{i=0}^{r-1}
\frac { 1 - q^{k(a-i)}X^k } { 1 - q^{k(i-a-1)}X^{-k} } = 
(-1)^r \cdot q^{n/2} \cdot (q^{1/2}X)^n \cdot
\frac { 1 - q^{-ka}X^k } { 1 - q^{n} q^{-ka}X^k }.
\end{equation}
Let us compute the $\e$-factor occurring on the right hand side of
\eqref{alanbeck}.
By \cite{BHBLMS99} Théorème~2
we~ha\-ve $\e(q^{-1/2},\rt,\rt^\vee,\psi)=\cen_{\rt}(-1)^{k-1}$,
where $\cen_{\rt}$ is the central~cha\-racter of $\rt$.
Thus
\begin{equation}
\label{C2}
\e(\X,\rt,\rt^\vee,\psi)
= \cen_{\rt}(-1)^{k-1} \cdot (q^{1/2}X)^{f}
\end{equation}
for some integer $f=f(\rt,\rt^\vee,\psi)\in\ZZ$.
Putting \eqref{C1} and \eqref{C2} together, 
we get
\begin{equation*}
\g(\X,\St_r(\rt),\rt^\vee,\psi) = 
(-1)^r \cdot q^{n/2} \cdot \cen_{\rt}(-1)^{r(k-1)} 
\cdot (q^{1/2}X)^{rf}
\cdot \frac { 1 - q^{-ka}X^k } { 1 - q^{n} q^{-ka}X^k }.
\end{equation*}
Since $q^n$ is congruent to $1$ mod $\ell$
(which follows from the fact that $\pi$ is cuspidal but non-supercus\-pi\-dal)
and $\cen_\rho^r$ is the central character of $\pi$, we get
\begin{equation*}
\label{facteurLnr}
\g(\X,\pi,\rho^\vee,\psi) = (-1)^r \cdot q^{n/2} 
\cdot \cen_{\pi}(-1)^{k-1} \cdot (q^{1/2}X)^{rf}.
\end{equation*}
Specializing at $X=q^{-1/2}$, we get the expected result.
\end{proof}

In conclusion,
in order to prove Theorem \ref{CONJNRV},
it remains to prove Pro\-po\-sitions 
\ref{padicfinite} and \ref{minelli}.

\section{Gamma factors over finite fields}
\label{duveyrier}

In this section,
$\k$ is a finite field of characteristic $p$, 
$\ell$~is a pri\-me number different from $p$
and~$\R$ is an algebraically closed field~of characteristic~$0$~or $\ell$.
We fix a non-trivial character
\begin{equation*}
\label{psikd}
\psi : \k \to \R^\times.
\end{equation*} 
The purpose of this section is to 
recall the definition and properties
of the $\g$-factors of \cite{Roditty,Nien,Mossetal}
and prove Proposition \ref{padicfinite}.

\subsection{}
\label{Appert}
\label{juliensorel}

Given an integer $n\>1$, 
we write $\G=\G_n$ for the group $\GL_n(\k)$ 
and  $\N=\N_n$ for the~sub\-group of $\G$~made~of all~up\-per triangular
unipotent matrices.
The cha\-rac\-ter
\begin{equation*}
x\mapsto\psi(x_{1,2}+\dots+x_{n-1,n})
\end{equation*}
of $\N$ will still be denoted by $\psi$.
Let~us introduce the following definitions.

\begin{defi} 
\label{defwtg}
\begin{enumerate}
\item 
A representation $\pi$ of $\G$ is \emph{of Whittaker type} if
$\dim_R\Hom_{\N}(\pi,\psi)=1$. 
\item
A representation of $G$ is said to be \textit{generic} if it is
both irreducible and of Whittaker type. 
\end{enumerate} 
\end{defi}

Since any two non-trivial characters of $N$ are conjugate under the
normalizer of $N$ in $G$,~being
of Whit\-taker type does not depend on the choice of $\psi$.

Since $N$ is a $p$-group and the characteristic of $R$ is different from $p$,
it follows that
\begin{itemize}
\item 
a representation of~Whit\-ta\-ker type has a unique generic subquotient, 
\item
a representation is of~Whit\-ta\-ker type if and only if its contragredient is
of~Whit\-ta\-ker type.
\end{itemize}
If $\pi$ is a representation of Whittaker type 
and $\xi\in\Hom_{\N}(\pi,\psi)$ is non-zero,
the map associating~to any vector $v$ in $\pi$
the matrix coefficient $\coef_{v,\xi}$ (see Paragraph \ref{par21})
is a morphism from $\pi$ to $\Ind_{\N}^{\G}(\psi)$
whose image $\Ww(\pi,\psi)$ does not depend~on the choice of $\xi$.
This image is called the \emph{Whittaker~mo\-del} of $\pi$ with respect to $\psi$.
If $\tau$ is the unique generic subquotient of $\pi$,
then the unique irreducible sub\-representation of $\Ww(\pi,\psi)$ is
equal to $\Ww(\tau,\psi)$.

As in the non-Archimedean case (see Paragraph \ref{discdepq12}),
we denote by $g^*$ the transpose of~the~in\-ver\-se of a~ma\-trix $g\in\G$
and by $\wn$ the antidiagonal~per\-mu\-ta\-tion ma\-trix of maximal length
of~$G$.
If $\pi$ is of Whittaker type, the representation
$\pi^* : g \mapsto \pi (g^*)$ 
is~of Whittaker type and the map
\begin{equation*}
W\mapsto \left(\widetilde{W} : g \mapsto W(\wn g^*) \right)
\end{equation*} 
is an isomorphism of $R$-vector spaces
from $\Ww(\pi,\psi)$ to $\Ww(\pi^*,\psi^{-1})$. 
Moreover,
the unique generic sub\-quo\-tient of $\pi^*$ is isomorphic to
that of $\pi^\vee$.

\subsection{}
\label{Bessel}

We introduce the Bessel function of a representation of Whittaker type. 

\begin{prop}
Let~$\pi$ be a representation of Whittaker type of~$\G$. 
There is a unique function $\Be_{\pi,\psi} \in \Ww(\pi,\psi)$ such that
\begin{enumerate}
\item $\Be_{\pi,\psi}(1)=1$,
\item $\Be_{\pi,\psi}(xgy)=\psi(xy) \Be_{\pi,\psi}(g)$ for all $g\in G$ and $x,y\in N$.
\end{enumerate}
It is called the \emph{Bessel function} of~$\pi$ with respect to $\psi$. 
\end{prop}

\begin{proof}
Since $\pi$ is of Whittaker type, 
its Whittaker model is of Whittaker type as well. 
Thus the space of functions $B\in\Ww(\pi,\psi)$
such that $B(gx)=\psi(x)B(g)$ has dimension $1$.
We are thus reduced to proving that $B(1)\neq0$
for at least one of such $B$.
For this,
let $V$ denote the space of $\pi$.
It decomposes as the direct sum of the line
$V^\psi$ made of all $v\in V$ such that
$\pi(x)(v) = \psi(x)v$ for all $x\in N$
and its unique $N$-stable complement $V(\psi)$.
Fix a non-zero $v\in V^\psi$,
and a non-zero linear form $\xi$ on $V$ with kernel $V(\psi)$. 
Then the function $B=\coef_{v,\xi}$ has the required property. 
\end{proof}

We collect a couple of properties of the Bessel function. 

\begin{prop}
\label{corsqsbessel}
Let $\pi$ be a representa\-tion of Whittaker type of~$\G$ 
and $\tau$ be its unique generic subquotient.
Then $\Be_{\pi,\psi}=\Be_{\tau,\psi}$.
\end{prop}

\begin{proof} 
We have
$\Be_{\tau,\psi} \in \Ww(\tau,\psi) \subseteq \Ww(\pi,\psi)$.
By uniqueness,
it follows that $\Be_{\tau,\psi} = \Be_{\pi,\psi}$.
\end{proof}

\begin{prop}
\label{BesselJell}
Let~$\widetilde{\pi}$ be a generic~$\qlb$-representation of~$\G$
and~$\pi$ be the unique generic~sub\-quo\-tient of its reduction mod $\ell$.
Then~$\Be_{\widetilde{\pi},\psi}$ takes values in~$\zlb$,
and its reduction mod the maximal ideal of $\zlb$
is equal to $\Be_{\pi,\psi}$.
\end{prop}

\begin{proof}
The formula given in \cite{Gelfand70} Proposition 4.5 for the Bessel function 
of~$\widetilde{\pi}$ shows that it~ta\-kes~values in $\zlb$.
Let $L$ be the lattice of~$\zlb$-valued Whittaker functions in 
$\Ww(\widetilde{\pi},\psi)$.
Then~$L\otimes\flb$ is~of Whittaker type and
its Bessel function is the reduction of $\Be_{\widetilde{\pi},\psi} \in L$ 
mod the maximal~ideal~of $\zlb$.
The
claimed equality follows from~Propo\-sition \ref{corsqsbessel}
applied to $L\otimes\flb$ and $\pi$. 
\end{proof}

Let $\P$ denote the standard mirabolic subgroup of $G$,
that is,
the subgroup made of all matrices with last row
$(0\ \dots\ 0\ 1)$.

\begin{lemm}
\label{suppbessel}
Let~$\pi$ be a representation of Whittaker type of~$\G$.
If $g\in P$ and $\Be_{\pi,\psi}(g)\neq0$,~then we have $g\in N$.
\end{lemm}

\begin{proof}
The proof for complex representations (see \cite{Gelfand70} Proposition 1.2)
applies to $R$-represen\-ta\-tions. 
\end{proof}

\subsection{}

Let $n,m$ be integers such that $1\<m<n$.
Given
representations of Whittaker type~$\pi$,~$\pi'$~of $\GL_n(\k)$, $\GL_m(\k)$,
respectively,
an integer $j$ such that $0\< j\< n-m-1$,
and Whittaker func\-tions~$W\in\Ww(\pi,\psi)$ and $W'\in\Ww(\pi',\psi^{-1})$,
we define the sum 
\begin{equation*}
\I_j(W,W')= \sum_{g\in \N_m\backslash \G_m}
\sum_{x\in \Mat_{j,m}(\k)}
W\begin{pmatrix}g&0&0\\x& 1_{j} &0\\0&0&1_{n-m-j}\end{pmatrix} 
W'(g).
\end{equation*}
We introduce the
Rankin--Selberg $\g$-factor of the pair $(\pi,\pi')$ following
\cite{Roditty} Theorems 5.1, 5.4,~\cite{Nien} Theorem 2.10 and
\cite{Mossetal} Corollary 3.1.2. 
We set $w_{m,n-m}={\rm diag}({\rm id}_m,w_{n-m})$ as in \S\ref{discdepq12}.
We also let $q$ denote the cardinality of $\k$ and fix a square root $q^{1/2}$
of $q$ in $R^\times$.

\begin{prop} 
\label{FERS} 
Let~$\pi$ be a cuspidal representation of $\GL_n(\k)$
and $\pi'$ be a~repre\-sentation~of Whittaker type of~$\GL_m(\k)$.
Let $\cen_{\pi'}$ denote the central character of the generic subquotient
of $\pi'$.
There exists a unique scalar $\g_{\RS}(\pi,\pi',\psi) \in R$ such that 
\begin{equation}
\label{FERS1} 
\I_{n-m-1-j}\left( w_{m,n-m}\cdot \widetilde{W},\widetilde{W'}\right) =
\cen_{\pi'}(-1)^{n-1} \cdot q^{\frac{m(n-m-1-2j)}{2}}
\cdot \g_{\RS}(\pi,\pi',\psi) \cdot \I_j(W,W')
\end{equation}
for all Whittaker functions $W\in \Ww(\pi,\psi)$, $W'\in \Ww(\pi',\psi^{-1})$
and all integers $j\in\{0,\dots,n-m-1\}$.
\end{prop}

\begin{proof} 
The case where $j=0$ is given by \cite{Mossetal} Corollary 3.1.2
(where the authors have a norma\-li\-zation of the $\g$-factor different from ours).
The functional equation for any~$j$ follows from there 
as in the proof of \cite{Nien} Theorem 2.10.
\end{proof}

\begin{rema}
Our choice of normalization of the $\g$-factor $\g_{\RS}(\pi,\pi',\psi)$
is dictated by~Proposi\-tion \ref{depthzerocompat} below. 
\end{rema}

We collect a couple of properties of the $\g$-factor. 

\begin{prop}
\label{corsqsgamma}
\label{RSBesseldef}
Let~$\pi$ be a cuspidal representation of $\GL_n(\k)$
and $\pi'$ be~a~repre\-sentation~of Whittaker type of~$\GL_m(\k)$.
\begin{enumerate}
\item 
The scalar $ \g_{\RS}(\pi,\pi',\psi)$ is non-zero. 
\item
We have
\begin{equation*}
\g_{\RS}(\pi,\pi',\psi) = 
\cen_{\pi'}(-1)^{n-1} \cdot q^{\frac{m(n-m-1)}{2}} \cdot 
\sum_{\g\in N_m\backslash \G_m} \Be_{\pi,\psi}\begin{pmatrix} 0 & 1_{n-m}\\
  g&0\end{pmatrix} \Be_{\pi',\psi^{-1}}(g).
\end{equation*}
\item 
If~$\tau$ is the unique generic subquotient of~$\pi'$, then
$\g_{\RS}(\pi,\pi',\psi)=\g_{\RS}(\pi,\tau,\psi)$.
\item
Given an $a\in\k^\times$, and denoting by $\psi^a$ the character
$x\mapsto\psi(ax)$ of $\k$, one has
\begin{equation}
\label{epspsiaff}
\g(\pi,\pi',\psi^a)=\cen_\pi(a)^m \cdot \cen_{\pi'}(a)^n
\cdot \g(\pi,\pi',\psi).
\end{equation} 
\end{enumerate}
\end{prop}

\begin{proof}
Assertion (2) follows from Lemma \ref{suppbessel},
as in \cite{Nien} Proposition 2.16.
Assertion (3) follows from Proposition \ref{corsqsbessel} and Assertion (1).
Assertion~(4) follows from Assertion (2) together with~the fact that
$\Be_{\pi,\psi^a}$ is equal to $g\mapsto\Be_{\pi,\psi}(tgt^{-1})$~where~$t$~is
the diagonal matrix ${\rm diag}(a^{n-1},\dots,a,1)$. 

Since $\pi^*$ is cuspidal and $\pi'^*$ is of Whittaker type,
we may apply
Pro\-po\-si\-tion \ref{FERS} to the~Whitta\-ker functions
$w_{m,n-m}\cdot \widetilde{W} \in \Ww(\pi^*,\psi^{-1})$, 
$\widetilde{W'} \in \Ww(\pi'^*,\psi)$ and the integer $n-m-1-j$.~Thus
\begin{equation*}
\I_{j}\left( {W},{W'}\right)
= \cen_{\pi'}(-1)^{n-1} \cdot q^{\frac{-m(n-m-1-2j)}{2}} \cdot 
\g_{\RS}(\pi^\vee,\pi'^*,\psi^{-1}) 
\cdot \I_{n-m-1-j}(w_{m,n-m}\cdot\widetilde{W},\widetilde{W}'). 
\end{equation*} 
Combining this identity with \eqref{FERS1}, we get
\begin{equation*}
\label{FEtwice}
\g_{\RS}(\pi,\pi',\psi) \cdot \g_{\RS}(\pi^*,\pi'^{*},\psi^{-1})=1
\end{equation*}
which proves that $\g_{\RS}(\pi,\pi',\psi)$ is non-zero.
\end{proof}

\begin{rema}
\label{Ass} 
When $m$ and $n$ are odd,
$\g_{\RS}(\pi,\pi',\psi)$ depends on the choice of~$q^{1/2}$. 
\end{rema}

Let us recall that
the reduction mod $\ell$ of a cuspidal $\qlb$-representation of
$\GL_n(\k)$ is (irreducible and) cuspidal (\cite{Vigbook} III.1.1).

\begin{prop}
\label{reductionofgamma}
Let~$\widetilde{\pi}$ be a cuspidal $\qlb$-representation of
$\GL_n(\k)$ and $\widetilde{\pi}'$ be~a generic~$\qlb$-re\-pre\-sen\-tation
of~$\GL_m(\k)$. 
Let $\pi$ be the reduction mod $\ell$ of $\widetilde{\pi}$
and $\pi'$ be the unique~ge\-ne\-ric~sub\-quo\-tient of
the reduction mod $\ell$ of $\widetilde{\pi}'$.
% $\r_{\ell}(\widetilde{\pi}')$.
Then the reduction mod $\ell$ of $\g_{\RS}(\widetilde{\pi},\widetilde{\pi}',\psi)$
is equal to $\g_{\RS}(\pi,\pi',\psi)$.
\end{prop}
 
\begin{proof}
This follows from Lemma \ref{BesselJell} and Proposition \ref{RSBesseldef}. 
\end{proof}

\subsection{}

We now state the multiplicativity property of $\g$-factors that we will need. 

\begin{lemm}
\label{LemmabasicsWTfams} 
Let $n_1,n_2$ be positive integers,
and let $\pi_i$ be a representation of Whittaker type of $\GL_{n_i}(\k)$,
for $i=1,2$.
Then $\pi_1\times\pi_2$ is a representation of Whittaker type of
$\GL_{n_1+n_2}(\k)$.
\end{lemm}

\begin{proof}
Let $U$ be unipotent radical of the standard parabolic subgroup of 
$G=\GL_{n_1+n_2}(\k)$~ge\-nerated by~the~stan\-dard Levi subgroup
$M=\GL_{n_1}(\k)\times\GL_{n_2}(\k)$ and the subgroup $N$ of
unipotent upper triangular matrices.
Then
\begin{eqnarray*}
\Hom_{N}(\pi_1\times\pi_2,\psi) &\simeq&
\Hom_{\G}(\pi_1\times\pi_2,\Ind_{\N}^{\G}(\psi)) \\ &\simeq&
  \Hom_{M}(\pi_1\otimes\pi_2,\Ind_{\N}^{\G}(\psi)^{U}). 
\end{eqnarray*}
The lemma follows from the fact that
$\Ind_{\N}^{\G}(\psi)^{U} \simeq \Ind_{\N\cap M}^{M}(\psi|_{\N\cap M})$
by \cite{BR06} Theorem 2.1. 
\end{proof}

\begin{lemm}
\label{Jacopo}
Let $\pi$ be a cuspidal irreducible representation of $\GL_n(\k)$ and
$\pi'$ be a~represen\-ta\-tion of Whittaker type of~$\GL_{m}(\k)$.
Let $S$ be an algebraically closed extension of $R$.
Then $\pi\otimes S$ is irreducible and cuspidal,
$\pi'\otimes S$ is of Whittaker type and 
\begin{equation}
\label{BCRS}
\g_{\RS}(\pi\otimes S,\pi'\otimes S,\psi\otimes S)=\g_{\RS}(\pi,\pi',\psi).
\end{equation}
\end{lemm}

\begin{proof}
The fact that $\pi\otimes S$ is irreducible and cuspidal
follows from \cite{Vigbook} Théorème III.1.1~toge\-ther~with the
fact that the base change from $R$ to $S$ of the mirabolic $R$-representation
is iso\-mor\-phic to the mirabolic $S$-representation.
The fact that $\pi'\otimes S$ is of Whittaker type follows~from the fact
that, if $\kappa$ is an irreducible component of the restriction of $\pi'$ to $N$, 
then $\kappa\otimes S$ is irreducible.
The identity \eqref{BCRS} follows from Proposition \ref{corsqsgamma}(1)
together with the fact that,
by uniqueness,~the Bessel functions of $\pi'$ and $\pi'\otimes S$ coincide.
\end{proof}

\begin{prop}
\label{mult} 
Let~$\pi$ be a cuspidal representation of $\GL_n(\k)$,
let $m_1$ and $m_2$ be positive integers such that $m_1+m_2=m$
and let $\pi_i$ be a representation~of Whittaker type
of~$\GL_{m_i}(\k)$ for $i=1,2$. 
Then 
\begin{equation*}
\g_{\RS}(\pi,\pi_1\times\pi_2,\psi) =
\g_{\RS}(\pi,\pi_1,\psi) \cdot \g_{\RS}(\pi,\pi_2,\psi).
\end{equation*}
\end{prop}

\begin{proof}
Suppose first that $R$ is the field $\qlb$.
Soudry-Zelingher \cite{SZ} defined a factor $\Ga(\pi,\pi',\psi)$
for any representation $\pi'$ of Whittaker type of $\GL_{m}(\k)$.
Comparing our Proposi\-tion \ref{RSBesseldef} with~\cite{SZ} Theorem 3.4(1), 
we deduce that
\begin{equation}
\label{comparaisonSZNRV}
\Ga(\pi,\pi',\psi) = \cen_{\pi'}(-1)^{n} \cdot q^{mn/2} \cdot \g_{\RS}(\pi,\pi'^*,\psi).
\end{equation}
(Note that the representation $\pi'$ is assumed to be irreducible in \cite{SZ} 
Theorem 3.4. 
However,~sin\-ce~the fac\-tor $\Ga(\pi,\pi',\psi)$ only depends
on the generic subquotient $\tau$ of 
$\pi'$ by definition,
the~iden\-ti\-ty of \cite{SZ} Theorem 3.4(1)
also makes sense for $\pi'$ not necessarily irreducible,
provided that~one replaces $\pi'^\vee$ by $\pi'^*$ in the right hand side.)
Putting \eqref{comparaisonSZNRV} together with
\cite{SZ} Theorem 3.3,~and~ob\-serving 
that $(\pi_1\times\pi_2)^*$ and $\pi_1^*\times\pi_2^*$ share 
the same generic subquotient, 
we deduce that Proposition \ref{mult} holds for $R=\qlb$.

Suppose now that $R$ is the field $\flb$.
For $i=1,2$,
fix a generic $\qlb$-representation $\widetilde{\pi}_i$ 
of~$\GL_{m_i}(\k)$ whose reduction mod $\ell$ contains the
generic subquotient $\tau_i$ of $\pi_i$
and fix a cuspidal $\qlb$-lift $\widetilde{\pi}$ of~$\pi$.
We have
\begin{equation}
\label{marmite1}
\g_{\RS}(\widetilde{\pi},\widetilde{\pi}_1\times\widetilde{\pi}_2,\psi) =
\g_{\RS}(\widetilde{\pi},\widetilde{\pi}_1,\psi) \cdot 
\g_{\RS}(\widetilde{\pi},\widetilde{\pi}_2,\psi)
\end{equation}
and,
by Proposition \ref{corsqsgamma}, 
the left hand side of \eqref{marmite1} is equal to 
$\g_{\RS}(\widetilde{\pi},\widetilde{\tau},\psi)$ where 
$\widetilde{\tau}$ is the generic~sub\-quotient of 
$\widetilde{\pi}_1\times\widetilde{\pi}_2$.
Applying Proposition \ref{reductionofgamma}, 
we get 
\begin{equation}
\label{marmite2}
\g_{\RS}({\pi},{\tau},\psi) =
\g_{\RS}({\pi},{\tau}_1,\psi) \cdot \g_{\RS}({\pi},{\tau}_2,\psi)
\end{equation}
where $\tau$ is the generic subquotient of the reduction mod $\ell$ of 
$\widetilde{\tau}$,
and the right hand side of~\eqref{marmite2}~is equal to
$\g_{\RS}({\pi},{\pi}_1,\psi) \cdot \g_{\RS}({\pi},{\pi}_2,\psi)$
by Proposition \ref{corsqsgamma}.
The expected result now follows from Proposition \ref{corsqsgamma} again, 
together with the fact that $\tau$ is also the generic subquotient of 
$\pi_1\times\pi_2$. 

Now assume that $R$ is an arbitrary algebraically closed field
  of characteristic $\ell$.
  Fix a field~em\-bed\-ding $\iota:\flb\to\R$.
There are a cuspidal $\flb$-representation $\pi_\ell$
and $\flb$-representations $\pi_{1,\ell}$, $\pi_{2,\ell}$~of Whittaker type
such that $\pi \simeq\pi_\ell\otimes\R$
and $\pi_i \simeq \pi_{i,\ell}\otimes\R$ for $i=1,2$, 
where tensor products are taken with respect to $\iota$.
The expected result now follows from Lemma \ref{Jacopo}. 
\end{proof}

\subsection{} 
\label{Preuvepadicfinite}

Let us prove  Proposition \ref{padicfinite}. 
We will actually prove the following more general statement.
Let $F$ be a non-Archimedean locally compact field with residue field $\k$. 
Given a cuspidal~re\-pre\-sentation of level $0$ of $\GL_n(F)$,
we associate with it a cuspidal representation $\overline{\pi}$~of $\GL_n(\k)$
as~in Para\-graph \ref{typelevel0},
called the \textit{type} of $\pi$.
{Let $\psi_F$ be a non-trivial character of $F$ of conductor $1$
such that the character of $\k$ it induces is equal to $\psi$.}

\begin{prop}
\label{depthzerocompat}
Let~$\pi$,~$\pi'$ be cuspidal representations of level $0$ of
$\GL_n(F)$, $\GL_m(F)$,~res\-pectively, with $n>m$.
Let $\overline{\pi}$ and $\overline{\pi}'$ be their types. 
Then
\begin{equation*}
\g(X,\pi,\pi',\psi_F)=\g_{\RS}(\overline{\pi},\overline{\pi}',\psi).
\end{equation*}
In particular,~the $\g$-factor $\g(X,\pi,\pi',\psi_F)$ is constant. 
\end{prop}

\begin{proof}
In the case when $\R=\qlb$,
it follows from \cite{NienZhang} Theorem 3.11 that 
\begin{equation*}
  \g(X,\pi,\pi',\psi_F)=\cen_{\pi'}(-1)^{n-1} \cdot q^{m(n-m-1)/2} \cdot
  \overline{\g}(\overline{\pi},\overline{\pi}',\psi),
\end{equation*}
where $\overline{\g}(\overline{\pi},\overline{\pi}',\psi)$ is the factor
associated with the pair $(\overline{\pi},\overline{\pi}')$
by \cite{Nien} Theorem 2.10.
Comparing~our Proposi\-tion \ref{RSBesseldef} with \cite{Nien} Proposition 3.6, 
we deduce that
\begin{equation*}
\label{comparaisonNienNRV}
\overline{\g}(\overline{\pi},\overline{\pi}',\psi) =
\cen_{\pi'}(-1)^{n-1} \cdot q^{-m(n-m-1)/2} \cdot 
\g_{\RS}(\overline{\pi},\overline{\pi}',\psi). 
\end{equation*}
Thus Proposition \ref{depthzerocompat} holds when $R=\qlb$.
From there,
we deduce Proposition \ref{depthzerocompat} when~$R=\flb$
by using Propositions \ref{gammaint} and \ref{reductionofgamma},
and we pass from $\flb$ to any algebraically closed field of~cha\-rac\-teristic 
$\ell$ as in the proof of Proposition \ref{mult}.
\end{proof}

In order to prove Theorem \ref{CONJNRV},
it now only remains to prove Proposition \ref{minelli}.

\subsection{}
\label{GJresults}

Before proceeding to the proof of Proposition \ref{minelli},
let~us collect results from \cite{NRVGJ}~which~we will need.
Recall that we have fixed a square root $q^{1/2}$ of $q$ in $R$.
Associated with any cuspidal~re\-presentation $\pi$ of $\GL_n(\k)$,
there is its~Godement--Jac\-quet $\g$-factor $\g(\pi,\psi)$
defi\-ned in \cite{Kondo,Roditty} for com\-plex representations
and in \cite{NRVGJ} for $R$-repre\-sen\-ta\-tions.

\begin{prop}
\label{RSGJ}
For any cuspidal representation $\pi$ of $\GL_n(\k)$, we have
\begin{equation*}
\g(\pi,1,\psi) = \g(\pi,\psi)
\end{equation*}
where $1$ in the left hand side is the trivial character of $\k^\times$.
\end{prop}

\begin{proof}
The point is that these two $\g$-factors are defined by two different 
functional equations,
so the fact that they are equal is not immediate. 
(See \cite{Roditty} Theorem 4.1.1 and {\cite{NRVGJ} Definition~2.12}
for the functional equations defining $\g(\pi,\psi)$.)
For $R=\qlb$, the equality follows from \cite{Roditty}~Theo\-rem 4.2.1.
We deduce it for $R=\flb$ by~Propo\-sition \ref{reductionofgamma} and
{\cite{NRVGJ} Proposition 6.1},
and for~ar\-bi\-trary $R$ as in the proof of Proposition \ref{mult}.
\end{proof}

Now assume that $H$ is one of the subgroups of $\GL_n(\k)$ of Proposition 
\ref{minelli},
that is:
\begin{itemize}
\item 
either $\k$ has a subfield $\k_0$ such that $[\k:\k_0]=2$ 
and $H=\GL_n(\k_0)$,
\item
or $n=2m$ for some integer $m\>1$ and $H$ is the standard
Levi subgroup $\GL_{m}(\k)\times\GL_{m}(\k)$.
\end{itemize}
Given an $H$-distingui\-shed cuspidal representation $\pi$ of $\GL_n(\k)$,
the space $\Hom_{H}(\pi,1)$ has dimen\-sion $1$
(\cite{VSANT19} Remark 4.3 in the Galois case,
and \cite{VSANT19} Corollary 2.16 in the Levi case).
In the Levi case,
the element
\begin{equation*}
s = 
\begin{pmatrix}0&1_m\\1_m&0\end{pmatrix} \in \GL_{n}(\k)
\end{equation*}
normalizes $H$.
It thus acts on this space by a sign,
which we denote by ${\rm sgn}(\pi) \in \{-1,1\}\subseteq R^\times$.~In
{\cite{NRVGJ} Corollary 5.2, Theorem 5.3} we proved:

\begin{prop}
\label{valueGJdist}
Let $\pi$ be an ${H}$-distingui\-shed cuspidal representation of
$\GL_n(\k)$.~{In~the~Ga\-lois case,
  assume that $q^{1/2}$ is the cardinality of $\k_0$ and that $\psi$
  is trivial on $\k_0$.}
Then
\begin{equation*}
\g(\pi,\psi) = \left\{ 
\begin{array}{ll}
1 & \text{in the Galois case,} \\
{\rm sgn}(\pi) & \text{in the Levi case.}
\end{array}\right.
\end{equation*}
\end{prop}

We now proceed to the proof of Proposition \ref{minelli} in the next section. 

\section{Gamma factors and distinction}
\label{luigivampa}

For any integer $n\>1$, let 
$G_n$~be~the group $\GL_n(\k)$,
$N_n$ be the subgroup~of~its up\-per~triangular unipotent matrices
and $P_n$ be its~mi\-ra\-bolic subgroup.
Let $H_n$ be any of the following sub\-groups of $G_n$:
\begin{itemize}
\item (Galois case)
the subgroup $\GL_n(\k_0)$ where $\k$ is a quadratic extension of $\k_0$, 
\item (Levi case)
the centralizer of the diagonal matrix
${\rm diag}(-1,1,\dots,(-1)^n)$ in $G_n$.~De\-pending on the~pa\-ri\-ty of 
$n$, we thus have 
\begin{equation*}
H_{n} = 
\begin{pmatrix}
* & 0 & * & \cdots & 0 & * & 0 \\
0 & * & 0 & \cdots & * & 0 & * \\
* & 0 & * & \cdots & 0 & * & 0 \\
\vdots & \vdots & \vdots & \ddots & \vdots & \vdots & \vdots \\
0 & * & 0 & \cdots & * & 0 & * \\
* & 0 & * & \cdots & 0 & * & 0 \\
0 & * & 0 & \cdots & * & 0 & * 
\end{pmatrix}
\ \text{($n$ even)}
\quad\text{or}\quad
H_{n} = 
\begin{pmatrix}
* & 0 & * & \cdots & * & 0 & * \\
0 & * & 0 & \cdots & 0 & * & 0 \\
* & 0 & * & \cdots & * & 0 & * \\
\vdots & \vdots & \vdots & \ddots & \vdots & \vdots & \vdots \\
* & 0 & * & \cdots & * & 0 & * \\
0 & * & 0 & \cdots & 0 & * & 0 \\
* & 0 & * & \cdots & * & 0 & * 
\end{pmatrix}
\ \text{($n$ odd)}.
\end{equation*}
Note that,
if $r$ is the smallest integer such that $2r\>n$, 
the subgroup $H_n$ is conjugate to the~stan\-dard Levi subgroup
$\GL_{r}(\k)\times\GL_{n-r}(\k)$ of $G$.
\end{itemize}

\subsection{}
\label{groenland}

The purpose of this section is to prove the following propo\-si\-tion.

\begin{prop}
\label{minelli3}
Let $\pi$ be a non-supercuspidal cuspidal representation of $G_n$,
with~as\-so\-ciated supercuspidal representation $\rho$. 
Suppose that $\pi$~is $H_n$-distingui\-shed
but has no lift distinguished~by $H_n$. 
Then,
for~any~non-trivial~cha\-rac\-ter $\psi$ of $\k$,
we have 
$\g(\pi,\rho^\vee,\psi) = 1$.
\end{prop}

This proposition is equivalent to Propo\-si\-tion \ref{minelli}.
This is clear in the Galois case.
In the~Levi case,
the fact that $\pi$~is $H_n$-distingui\-shed implies that $n=2m$
for some $m\>1$ (by \cite{VSANT19}~Proposition 2.14).
Then, as $H_n$ is conjugate to the standard
Levi subgroup $G_m\times G_m$ in $G_n$,
distinction by~$H_n$ is equivalent to distinction by $G_m\times G_m$. 

Remark that,
for any $a\in\k^\times$,
one has $\g(\pi,\rho^\vee,\psi^a) = \g(\pi,\rho^\vee,\psi)$
from \eqref{epspsiaff}.
It thus~suffices~to prove Proposition \ref{minelli3} for a specific character $\psi$.
\textit{We will thus assume from now on until~the end~of this section, that, 
in the Galois case, the character $\psi$ is trivial on $\k_0$.}
In the Levi case,~we~make no~assumption on $\psi$.

\subsection{}

In this paragraph,
we collect some properties of the groups $H_n$.
First
we observe that~$H_n$~is the group of fixed points of an involution $\s$
of $G_n$.
In the Galois case,
it is the non-trivial~auto\-mor\-phism of $\k/\k_0$ acting componentwise.
In the Levi case
it is the~matrix
${\rm diag}(-1,1,\dots,(-1)^n)$
acting by conjugacy. 
Thanks to the assumption made on $\psi$ in the previous paragraph, 
the~char\-acter $\psi$ of $N_n$ 
satisfies $\psi\circ\s=\psi^{-1}$, which implies that 
\begin{equation}
\label{kaki1}
\text{$\psi$ is trivial on $N_n \cap H_n$.}
\end{equation} 
This explains why we didn't choose for $H_n$ a standard Levi subgroup in the
Levi case.

Let $(n_1,\dots,n_r)$ be a family of positive integers whose sum is equal to
$n$, 
let $M$ be the standard Levi subgroup $G_{n_1}\times\dots\times G_{n_r}$
of $G_n$,
let $Q$ be the standard parabolic subgroup of $G_n$ with Levi subgroup
$M$ and $Q^-$ be the opposite parabolic subgroup with respect to $M$.
Finally,
let $U$, $U^-$ be the unipotent radicals of $Q$, $Q^-$,
respectively.
Then
\begin{eqnarray}
\label{kakiIw1}
U^-MU\cap H_n &=& (U^-\cap H_n)(M\cap H_n)(U\cap H_n),\\
\label{kakiIw2}
M\cap H_n &=& H_{n_1} \times\dots\times H_{n_r}.
\end{eqnarray}

Recall that $g^*$ denotes the transpose of~the inverse of~a ma\-trix 
$g\in\G_n$ and $\wn$ is the antidiago\-nal
per\-mu\-ta\-tion~matrix of maximal length of $G_n$.
Then
\begin{eqnarray} 
\label{kaki3}
\text{$\H_n$ is stable by $*$}, \\
\label{kaki4}
\text{$\H_n$ is normalized by $\wn$}. 
\end{eqnarray}
This list of properties \eqref{kaki1} to \eqref{kaki4} will allow us to treat 
the Galois and Levi cases uniformly.

We will also need the following properties of $H_n$-distinguished generic 
representations of $G_n$.

\begin{lemm}
\label{malebranche}
For any $\H_n$-distinguished cuspidal representation $\pi$ of~$\G_n$, 
one has
\begin{equation*}
\dim \Hom_{\P_n\cap\H_n}(\pi,1) = 1.
\end{equation*}
\end{lemm}

\begin{proof} 
In the Levi case, the result is given by \cite{VSANT19} Remark 2.15.
In the Galois case,
is it given by \cite{AnandMatringe} Proposition 4.3 for $R=\CC$ only.
Let us consider the case of a general $R$ in the Galois case.
We will write $G=G_n$, $P=P_n$, etc.

First note that restriction of $\pi$ to $P$ is isomorphic to the induced 
representation $\Ind^P_N(\psi)$~(see \cite{Vigbook} III.1.1).
It follows from a simple application of Mackey's formula 
that the dimension of~the space $\Hom_{P\cap H} (\pi,R)$ is the
number of $(H,N)$-double cosets $HgN \subseteq G$ such that
$\psi |_{N \cap H^g}= 1$.

It follows that,
in the case when $R$ is the field of complex numbers,
\cite{AnandMatringe} Proposition 4.3 tells~us that 
there is only one double coset $HgN \subseteq G$ such that
$\psi |_{N \cap H^g}= 1$ (and it is $HN$).

Since $N$ is a $p$-group,
the character $\psi$ takes values in $\mu_{p^\infty}(R)$,
the group of elements of $R^\times$~who\-se order is a $p$-power.
Fix a group isomorphism
\begin{equation*}
\iota : \mu_{p^\infty}(R) \to \mu_{p^\infty}(\CC)
\end{equation*}
(which exists since $R$ is algebraically closed and has characteristic
different from $p$).
Then $\iota\circ\psi$~is well-defined and $HN$ is the
unique double coset $HgN \subseteq G$ such that 
$\iota\circ\psi |_{N \cap H^g}= 1$. 
Since $\iota$ is in\-jective,
the same result holds for $\psi$ itself.
We thus get the expected result. 
\end{proof}

\begin{lemm}
\label{distselfdual}
Any $H_n$-distinguished generic representation of $G_n$ 
is $\s$-self-dual.
\end{lemm}

\begin{proof}
As in the proof of the preceding lemma,
we write $G=G_n$, $H=H_n$, etc. 
In~the~Galois case,
any $H$-distinguished~ir\-re\-ducible representation
is $\s$-self-dual (see \cite{VSANT19} Remark 4.3). 
Suppo\-se that we are in the Levi case. 
Since $\s$ is an~inner involution,
an ir\-re\-ducible representation of~$G$
is $\s$-self-dual if and only if it is self-dual.
Let $\pi$ be an $H$-distinguished generic $R$-representation~of $G$. 
Arguing~as in \cite{NRV25} 5.7, we may and will assume that $R=\flb$.

By \cite{NRV25} Lemma 5.8
(whose proof works for any $H$-distinguished irreducible representation~of~$G$
and not only for $H$-distinguished cuspidal representations),
there exists an $H$-distinguished~irre\-du\-cible $\qlb$-representation
$\widetilde{\pi}$ of $G$ whose~re\-duc\-tion mod $\ell$ contains $\pi$.
The main result of \cite{Kapon} tells us that 
this representation $\widetilde{\pi}$ is self-dual.

Let $\mu_1,\dots,\mu_t$ be cus\-pidal $\qlb$-representations of
$\GL_{n_1}(\k),\dots,\GL_{n_t}(\k)$ respectively,
for~integers $n_1,\dots,n_t\>1$ of sum $n$,
such that $\widetilde{\pi}$ occurs in 
$\mu_1\times\dots\times\mu_t$.
By \cite{Vigbook} III.1.1,
the reduction mod $\ell$ of~$\mu_i$
is irreducible~and cuspi\-dal.
The representation $\pi$ is thus the unique generic subquotient~of the 
induced representation 
$\r_\ell(\mu_1)\times\dots\times \r_\ell(\mu_t)$.  
It is \textit{a fortiori} the unique generic subquotient of $\r_\ell(\widetilde{\pi})$.
Since $\widetilde{\pi}$ is self-dual,
$\pi$ and $\pi^{\vee\s}$ are generic subquotients of $\r_\ell(\widetilde{\pi})$.
Thus $\pi$ is self-dual. 
\end{proof}

\begin{lemm}
\label{distparitek} 
Let $\pi$ be a supercuspidal representation of $G_n$.  
If $\pi$ is $\s$-self-dual, then 
\begin{enumerate}
\item 
In the Galois case,
$n$ is odd and $\pi$ is $H_n$-distinguished.
\item 
In the Levi case, 
either $n$ is even and $\pi$ is $H_n$-distinguished, 
or $n=1$ and $\pi$ is a character~of order at most $2$ of $\k^\times$. 
\end{enumerate}
\end{lemm}

\begin{proof}
In the Galois case, 
this is \cite{NRV25} Lemmas 2.3 and 2.5.
In the Levi case, 
this is \cite{NRV25}~Lem\-mas 2.17 and 2.19. 
\end{proof}

\subsection{}
\label{floquage}

In this paragraph,
we fix an integer $n$ and abbreviate $G=G_n$, $P=P_n$, 
etc. 
Let $\U$ be the unipotent radical of $P$ 
and $\G'$ be the image of $\G_{n-1}$ in $\G$
via the~embed\-ding $g\mapsto{\rm diag}(g,1)$.
We thus have $P=\G'\U$.
We also write $N'=N\cap\G'$ and $H'=H\cap\G'$.
From \eqref{kakiIw1},~\eqref{kakiIw2}~one~has
\begin{equation}
\label{kakiMi}
\P\cap\H = \H'(\U\cap \H).
\end{equation}
The goal of this paragraph is to associate a non-zero scalar $c(\pi,\psi)$ 
to any $H$-distinguished~cus\-pidal representation $\pi$ of $G$,
which is a proportionality constant between two explicit $H$-invariant
linear forms on the Whittaker model of $\pi$
(see Proposition \ref{BernsteinConstant}).

Let $\pi$ be a representation of Whittaker type of~$\G$.
Associated with it,
there are $\R$-linear forms $\La=\La_{\pi}$ and $\Lat=\Lat_{\pi}$ 
on its Whittaker space $\Ww(\pi,\psi)$ defined
for all $W\in\Ww(\pi,\psi)$ by
\begin{eqnarray*}
\La(W) &=& \sum\limits_{h' \in \N'\cap\H' \backslash \H'} W(h'), \\
\Lat(W) &=& \sum\limits_{h' \in \N'\cap\H' \backslash \H'} \widetilde{W}(h') 
\ \ = 
\sum\limits_{h' \in \N'\cap\H' \backslash \H'} W(\wn h'^*).
\end{eqnarray*}
The first sum over $\N'\cap\H' \backslash \H'$
is well-defined thanks to \eqref{kaki1}. 
For the second sum,
observe that the conjugate of $(\N'\cap\H')^*$ by $\wn$
is contained in
$\wn^{\phantom{1}} (\N\cap\H)^*\wn^{-1} =
\N \cap \wn^{\phantom{1}} \H^* \wn^{-1}$, 
which~is equal to $\N \cap H$ thanks to \eqref{kaki3} and \eqref{kaki4}. 

\begin{lemm}
\label{LaPHinv}
The linear form~$\La$ is~$\P\cap \H$-invariant and non-zero. 
\end{lemm}

\begin{proof}
It follows from \eqref{kaki1} and \eqref{kakiMi} that,
for all $W\in\Ww(\pi,\psi)$,
we have
\begin{equation*}
\sum\limits_{g\in \P\cap\H} W(g)
= \sum\limits_{u\in\U\cap\H} \sum\limits_{h'\in\H'} W(uh')
= |\U\cap\H| \cdot |\N'\cap\H'| \cdot \La(W)
\end{equation*}
and the left hand side is $\P\cap \H$-invariant as a linear form on
$\Ww(\pi,\psi)$.~By Lemma~\ref{suppbessel},~we have
\begin{equation*}
\La(\Be_{\pi,\psi}) = \Be_{\pi,\psi}(1) = 1
\end{equation*}
hence~$\La$ is non-zero.
\end{proof}

\begin{lemm}
\label{LatPHinv}
The linear form $\Lat$ is~$(\P\cap \H)^*$-in\-variant and non-zero.
\end{lemm}

\begin{proof}
First,
it follows from \eqref{kakiMi} that 
\begin{equation*}
(\P\cap \H)^* = (\H'(\U\cap\H))^* = \H'^*(\U\cap\H)^*
\end{equation*}
and $\Lat$ is clearly invariant by $\H'^*$. 
It thus remains to prove that it is invariant by $(\U\cap\H)^*$.~For 
any $x\in(\U\cap\H)^*$,
we have
\begin{equation*}
\Lat(x \cdot W) = \sum\limits_{h' \in \N'\cap\H' \backslash \H'} W(\wn h'^* x).
\end{equation*}
For a given $h'\in\H'$,
write $u=\wn^{\phantom{1}} h'^* x h'^{*-1}\wn^{-1}$.
Since $H'$ normalizes $U\cap H$,
we have
\begin{equation*}
  u \in \wn^{\phantom{1}} (U\cap H)^*\wn^{-1} \subseteq
  \wn^{\phantom{1}} (N\cap H)^*\wn^{-1}
= \N \cap \wn^{\phantom{1}} \H^* \wn^{-1} = \N \cap \H.
\end{equation*}
It thus follows from \eqref{kaki1} 
that $W(\wn h'^* x)=W(u\wn h'^*)=\psi(u)W(\wn h'^*)=W(\wn h'^*)$.
The~ex\-pected result follows. 

Now let us consider the Bessel function 
$\Be_{\widetilde{\pi},\psi^{-1}}\in\Ww(\widetilde{\pi},\psi^{-1})$
and its~ima\-ge 
$\widetilde{\Be}_{\widetilde{\pi},\psi^{-1}} \in \Ww({\pi},\psi)$
by the map $W \mapsto\widetilde{W}$ defined in Paragraph \ref{juliensorel}. 
By Lemma~\ref{suppbessel},~we have
\begin{equation*}
  \Lat(\widetilde{\Be}_{\widetilde{\pi},\psi^{-1}}) = 
  \sum\limits_{h' \in \N'\cap\H' \backslash \H'} 
  \Be_{\widetilde{\pi},\psi^{-1}}(h') = \Be_{\widetilde{\pi},\psi^{-1}}(1) = 1 
\end{equation*}
hence~$\Lat$ is non-zero.
\end{proof}

\begin{lemm}
\label{BernsteinConstantGeneric}
Let~$\pi$ be an $\H$-distinguished {generic} representation of~$\G$.
Suppose that 
\begin{equation*}
\dim \Hom_{\P\cap\H}(\pi,1) = \dim \Hom_{\P\cap\H}(\pi^\vee,1) = 1.
\end{equation*}
Then there exists~a unique non-zero scalar $c(\pi,\psi)\in\R^\times$ such that
$\Lat = c(\pi,\psi) \cdot \La$.
\end{lemm}

\begin{proof}
We will identify~$\pi$ with its Whittaker model $\Ww(\pi,\psi)$.
By assumption, the containment
\begin{equation*}
\Hom_{\H}(\pi,1) \subseteq \Hom_{\P\cap\H}(\pi,1)
\end{equation*}
is an equality. 
The linear form~$\La$,
which is~$\P\cap\H$-invariant and non-zero by Lemma \ref{LaPHinv},
is thus $\H$-invariant. 

On the other hand,
the fact that the spaces $\Hom_{\P\cap\H}(\pi^*,1)$ and
$\Hom_{(\P\cap\H)^*}(\pi,1)$ are isomorphic together with the fact that
$\pi^*\simeq\pi^\vee$ imply that 
\begin{equation*}
\dim \Hom_{(\P\cap\H)^*}(\pi,1) = \dim \Hom_{\P\cap\H}(\pi^\vee,1) = 1.
\end{equation*}
Since $(P\cap H)^*=P^*\cap H^* \subseteq\H^*=\H$,
where the latter equality follows from \eqref{kaki3},
we have 
\begin{equation*}
\Hom_{\H}(\pi,1) = \Hom_{\H^*}(\pi,1) \subseteq \Hom_{(\P\cap\H)^*}(\pi,1)
\end{equation*}
and the containment is an equality. 
The form~$\Lat$,
which is~$(\P\cap\H)^*$-invariant and non-zero~by Lemma \ref{LatPHinv},
is thus $\H$-invariant. 
Since the space $\Hom_{\H}(\pi,1)$ has dimension $1$ by assumption,~it
follows that 
there exists~a unique non-zero scalar $c(\pi,\psi)\in\R^\times$ such that
$\Lat  = c(\pi,\psi) \cdot \La$.
\end{proof}

\begin{prop}
\label{BernsteinConstant}
For any $\H$-distinguished cuspidal representation $\pi$ of~$\G$,
there is a unique non-zero scalar $c(\pi,\psi)\in\R^\times$ such that
$\Lat = c(\pi,\psi) \cdot \La$.
\end{prop}

\begin{proof}
As $\pi^\vee$ is isomorphic to $\pi^*$ and $H^*=H$,
the spaces $\Hom_{H}(\pi^\vee,1)$ and $\Hom_{H}(\pi,1)$~have
the same dimension.
The representation $\pi^\vee$ is thus both cuspidal and $\H$-distinguished.
The~pro\-po\-si\-tion now follows from Lemmas \ref{BernsteinConstantGeneric} and 
\ref{malebranche}.
\end{proof}

\subsection{}

In preparation of the next paragraph, 
we prove the following lemma. 
As in the previous~pa\-ragraph, 
we fix an integer $n$ and abbreviate $G=G_n$, $H=H_n$, etc.

Let $RG$ denote the space of $R$-valued functions on $G$ equipped with
the action of $G$ by~right trans\-lations. 
Let $\Ii_\psi$ be the idempotent endo\-mor\-phism of $RG$ defined by
\begin{equation*}
\Ii_{\psi} (f)(g) = \frac 1 {|\N|} \sum_{u\in\N} \psi^{-1}(u)f(ug),
\quad
f\in RG, \ g\in G.
\end{equation*}
Its image is equal to the induced representation $\Ind_{\N}^{\G}(\psi)$.
We denote by $\Ind^G_H(1)$ the representa\-tion of $G$ induced by
the~tri\-vial~character of $H$,
whose space is made of all functions of $RG$ which are invariant by
left~trans\-lations by $H$.

\begin{lemm}
\label{keyabstractlemma}
Given~any $W\in\Ind_\N^\G(\psi^{-1})$,
the following assertions are equivalent.
\begin{enumerate}
\item 
For all $g \in\G$, one has
\begin{equation*}
\sum_{h\in \H} W(hg) = 0. 
\end{equation*}
\item
For all functions $\phi\in\Ii_{\psi}(\Ind^G_H(1))$, one has 
\begin{equation*}
\sum\limits_{g \in G} \phi(g) W(g) = 0. 
\end{equation*}
\end{enumerate} 
\end{lemm}

\begin{proof}
Given any $W\in\Ind_\N^\G(\psi^{-1})$ and $f\in\Ind^G_H(1)$, 
we have
\begin{eqnarray*}
\sum\limits_{g \in G} \Ii_{\psi}(f)(g) W(g) &=& 
\sum\limits_{g \in G} \frac 1 {|\N|} 
\sum_{u\in\N} \psi^{-1}(u) f(ug) W(g) \\
&=& \frac 1 {|\N|} \sum\limits_{g \in G} 
\sum_{u\in\N} f(ug) W(ug) \\
&=& \sum\limits_{g \in G} f(g) W(g) \\
&=& \sum\limits_{x \in H\backslash\G} f(x) \bigg(\sum_{h\in \H} W(hx)\bigg).
\end{eqnarray*}
The lemma follows.
\end{proof}

\subsection{}
\label{1966anneemirifique}

The following proposition will be crucial to our proof of 
Proposition \ref{minelli3}.
In this proposition, we introduce the technical assumption \eqref{saintepelagie}
which we will discuss in \S\ref{montdor1} and \S\ref{montdor2}. 

\begin{prop}
\label{Lemma1RSconstant}
Let~$\pi$ be an~$\H_n$-distinguished cuspidal representation 
of~$\G_n$
and let~$\pi'$~be a representation of Whittaker type of~$\G_{n-1}$
such that
\begin{equation}
\label{saintepelagie}
\Ww(\pi',\psi^{-1}) \cap
\Ii_{\psi^{-1}}\left(\Ind_{\H_{n-1}}^{\G_{n-1}}(1)\right)
\neq \{0\}.
\end{equation}
Then~$\g_{\RS}(\pi,\pi',\psi) = c(\pi,\psi)$.
\end{prop}

\begin{proof}
By the functional equation,
for all~$W\in\Ww(\pi,\psi)$ and all~$W'\in\Ww(\pi',\psi^{-1})$,
we have
\begin{equation*}
\sum_{g\in \N_{n-1}\backslash \G_{n-1}}
W\begin{pmatrix} 0&1\\g&0\end{pmatrix} W'(g) = 
\g_{\RS}(\pi,\pi',\psi) \cdot \sum_{g\in \N_{n-1}\backslash \G_{n-1}}
W\begin{pmatrix} g&0\\0&1\end{pmatrix} W'(g).
\end{equation*}
Let ${\it\Delta}$ denote the function 
\begin{equation*}
{\it\Delta}(g) = W\begin{pmatrix}0&1\\g&0\end{pmatrix} - 
c(\pi,\psi) \cdot W \begin{pmatrix} g&0\\0&1\end{pmatrix}
\end{equation*}
for all $g\in G_{n-1}$.
It belongs to the space $\Ind^{G_{n-1}}_{N_{n-1}}(\psi)$.
By Proposition \ref{BernsteinConstant},
we have
\begin{equation*}
\sum\limits_{h \in H_{n-1}} {\it\Delta}(h) = |N'\cap\H'| (\Lat(W) - c(\pi,\psi) \La(W)) = 0.
\end{equation*}
By Lemma \ref{keyabstractlemma} applied to $G_{n-1}$, $H_{n-1}$ and
$\psi^{-1}$, 
we thus have
\begin{equation*}
\sum_{g\in\N_{n-1}\backslash \G_{n-1}}\left(
W\begin{pmatrix}0&1\\g&0\end{pmatrix} - 
c(\pi,\psi) W \begin{pmatrix} g&0\\0&1\end{pmatrix}
\right) \phi(g) = 0
\end{equation*} 
for all~$\phi\in \Ii_{\psi^{-1}}\left(\Ind_{\H_{n-1}}^{\G_{n-1}}(1)\right)$.
Taking a non-zero $W'$ in the left hand side of
\eqref{saintepelagie}, 
we have 
\begin{equation*}
\sum_{g\in\N_{n-1}\backslash \G_{n-1}}W\begin{pmatrix} 0&1\\g&0\end{pmatrix} W'(g) 
= c(\pi,\psi)\sum_{g\in\N_{n-1}\backslash \G_{n-1}}
W\begin{pmatrix} g&0\\0&1\end{pmatrix} W'(g).
\end{equation*}
Choosing a $g_0\in G_{n-1}$ such that $W'(g_0)\neq0$,
and setting $W=g_0^{-1} \cdot \Be_{\pi,\psi}$,
we get
\begin{equation*}
  \sum_{g\in\N_{n-1}\backslash \G_{n-1}}
  W\begin{pmatrix} g&0\\0&1\end{pmatrix} W'(g) =
  \sum_{g\in\N_{n-1}\backslash \G_{n-1}}
  \Be_{\pi,\psi}\begin{pmatrix} gg_0^{-1}&0\\0&1\end{pmatrix} W'(g) =
W'(g_0) \neq 0.
\end{equation*} 
As $\g_{\RS}(\pi,\pi',\psi)$ is the unique scalar which satisfies this
equation, we get~$\g_{\RS}(\pi,\pi',\psi)=c(\pi,\psi)$.
\end{proof}

\begin{coro}
\label{cabrespine}
Let $\pi$ be an $H_n$-distingui\-shed 
cuspidal representation~of $G_n$.
Let $\pi'$ and~$\pi''$ be $H_{n-1}$-distingui\-shed re\-pre\-sentations 
of Whittaker type of $G_{n-1}$.
Sup\-pose
that $\Ww(\pi',\psi^{-1})$ and $\Ww(\pi'',\psi^{-1})$ have a
non-zero intersection with
$\Ii_{\psi^{-1}}(\Ind_{\H_{n-1}}^{\G_{n-1}}(1))$.
Then
\begin{equation*}
\g(\pi,\pi',\psi) = \g(\pi,\pi'',\psi).
\end{equation*}
\end{coro}

\begin{proof}
This follows from the equalities
$\g(\pi,\pi',\psi) = c(\pi,\psi) = \g(\pi,\pi'',\psi)$
given by~Pro\-position \ref{Lemma1RSconstant} applied to $\pi'$ and $\pi''$.
\end{proof}

\subsection{}
\label{montdor1}

As in \S\ref{floquage},
the integer $n$ is fixed and we abbreviate $G=G_n$, $H=H_n$, etc. 
Let~us~in\-tro\-du\-ce the following definition. 

\begin{defi}
\label{WTspecialH}
A representation $\pi$ of Whittaker type of $G$ is said to be 
$H$-\textit{special},
or \textit{special} with respect to $H$,
if the intersection $\Ww(\pi,\psi) \cap \Ii_{\psi}(\Ind^G_H(1))$
is non-zero.
\end{defi}

We need to produce enough $H$-special representations of Whittaker type
in order to prove~Pro\-position \ref{minelli3}.
For this,
we introduce the following criterion. 

\begin{defi}
\label{Renal}
An $H$-distinguished representation $\pi$ of Whittaker type of $\G$ is
said to~be \textit{of class} 
$\Cc(H)$ if there is an $H$-invariant linear form on $\pi$ which is
non-zero on the $1$-dimensional space $\Hom_{\N}(\psi,\pi)$ of
vectors $v$ of $\pi$ such that $\pi(u)v=\psi(u)v$ for all $u\in N$.
\end{defi}

This gives us a sufficient condition for being $H$-special.

\begin{lemm}
\label{basicCpsilemma}
Any representation of $\G$ of class $\Cc(H)$ is $H$-special.
\end{lemm}

\begin{proof}
Let $\xi$ be an $H$-invariant linear form on $\pi$.
Given any vector $v\in\pi$,
the matrix coefficient $\coef_{v,\xi}$ is in $\Ind^G_H(1)$~and the function 
$\Ii_{\psi} (\coef_{v,\xi})$ is equal to $\coef_{v,\xi'}$ where 
\begin{equation*}
\xi' = \frac 1 {|N|} \sum\limits_{u\in\N} \psi(u) \xi \circ \pi(u^{-1})
\in \Hom_{N}(\pi,\psi).
\end{equation*}
If $\xi'$ is non-zero,
the image of $v\mapsto \coef_{v,\xi'}$ is $\Ww(\pi,\psi)$,
thus $\Ww(\pi,\psi)$
is contained in $\Ii_{\psi}(\Ind^G_H(1))$.

Now let $j\in\Hom_N(\psi,\pi)$ be non-zero and
suppose that $\xi(j)$ is non-zero.
One has
\begin{equation*}
\xi'(j) = \frac 1 {|N|} \sum\limits_{u\in\N} \psi(u) \xi(\pi(u^{-1})j) 
= \frac 1 {|N|} \sum\limits_{u\in\N} \xi(j)
= \xi(j)
\end{equation*}
which is non-zero.
\end{proof}

\begin{rema}
In the case when $R=\CC$,
it follows from \cite{AnandMatringe} Theorem 1.1 that,
in the Galois case,
any $\H$-distinguished generic representation of $\G$
is of class $\Cc(H)$.
\end{rema}

Recall that $P$ is the mirabolic subgroup of $G$.

\begin{lemm}
\label{cuspsatcpsi}
Let $\pi$ be an~$\H$-distinguished
generic representation of~$\G$.
Suppose that the~spa\-ce $\Hom_{\P\cap\H}(\pi,1)$ has dimension $1$.
Then~$\pi$ is of class $\Cc(H)$.
\end{lemm}

\begin{proof}
We identify~$\pi$ with its Whittaker model~$\Ww(\pi,\psi)$.
By assumption, the containment
\begin{equation*}
\Hom_{\H}(\pi,1) \subseteq \Hom_{\P\cap\H}(\pi,1)
\end{equation*}
is an equality. 
Let $\xi$ be the linear form on $\Ww(\pi,\psi)$ defined by
\begin{equation*}
\xi(W) = \sum_{h\in \P\cap \H} W(h).
\end{equation*}
It is $\P\cap\H$-invariant, thus $\H$-invariant.
Evaluating it on the Bessel function~$\Be_{\pi,\psi}$,
it follows from Lemma \ref{suppbessel} that
\begin{equation*}
  \xi(\Be_{\pi,\psi}) = \sum_{h\in \N\cap \H} \Be_{\pi,\psi}(h)
  = |\N\cap \H| \neq0.
\end{equation*}
Hence~$\pi$ is of class $\Cc(H)$.
\end{proof} 

It follows immediately from Lemma \ref{malebranche} that:

\begin{coro}
\label{CpsiCusp}
Any $H$-distinguished cuspidal representation of $\G$ is of class $\Cc(H)$.
\end{coro}

\subsection{}
\label{montdor2}

We now prove a preservation property of $\Cc(H)$ under parabolic
induction.
Let $n_1$ and $n_2$~be positive integers such that $n=n_1+n_2$. 
Let $M$ be the standard Levi subgroup~$G_{n_1}\times G_{n_2}$~of~$G_n$
and $Q$ be the standard parabolic subgroup of $G$ generated by $M$ and $N$. 
Let $U$ be the unipotent radical of $Q$ and $U^-$ be the
unipotent~ra\-di\-cal of 
the~para\-bo\-lic~subgroup $Q^-$ opposite to $Q$ with respect to $M$.
We will write $\psi_i$ for the character of $N_{n_i}$ induced by $\psi$.

\begin{lemm}
\label{cuspsatcpsi2} 
For $i=1,2$, let $\pi_i$ be an $H_{n_i}$-distinguished representation of 
Whittaker type~of class $\Cc(H_{n_i})$ of $\G_{n_i}$.  
The induced representation $\pi_1\times\pi_2$ is of class $\Cc(H_n)$.
\end{lemm}

\begin{proof}
Let $\pi$ be the representation of $G_n$ parabolically induced from
$\pi_M=\pi_1\otimes\pi_2$ along~the parabolic subgroup $Q^-$.
By \cite{HL} Theorem 1.1, it is isomorphic to $\pi_1\times\pi_2$.
It thus suffices~to~pro\-ve~that $\pi$ is of class $\Cc(H)$.
First,
a~sim\-ple application of Mackey's formula gives us
\begin{equation*}
\label{verrieres}
\Hom_{N}(\psi,\pi) \simeq
\bigoplus\limits_{g\in N\backslash G/Q^-} \Hom_{N^{g} \cap Q^-}(\psi^{g},\pi_{Q^-}) 
\end{equation*}
where $\pi_{Q^-}$ is the inflation of $\pi_M$ to $Q^-$.
By Lemma \ref{LemmabasicsWTfams},
we know that $\pi_1\times\pi_2$ (thus $\pi$) is~a~re\-presentation
of Whittaker type.
There is thus a unique double coset $NgQ^-$ contributing to this 
decomposition. 
For $g\in NQ^-$,
we have $N \cap Q^- = N \cap M = N_{n_1} \times N_{n_2}$, thus
\begin{equation*}
\Hom_{N \cap Q^-}(\psi,\pi_{Q^-}) \simeq
\Hom_{N_{n_1}}(\psi_1 ,\pi_1) \otimes \Hom_{N_{n_2}}(\psi_2,\pi_2) 
\end{equation*}
which has dimension $1$. 
Thus $NQ^-$ is the only contributing double coset
and one has an isomor\-phism of $R$-vector spaces
\begin{equation*}
\label{Doubs}
\Hom_{N}(\psi,\pi) \simeq 
\Hom_{N_{n_1} }(\psi_1,\pi_1) \otimes \Hom_{N_{n_2}}(\psi_2,\pi_2)
\end{equation*}
given by $f \mapsto f(1)$. 

For~$i=1,2$,
fix a non-zero vector $j_i \in \Hom_{N_{n_i}}(\psi_i,\pi_i)$,
and let $j \in \Hom_{N}(\psi,\pi)$ be the unique ele\-ment of $\pi$
such that $j(1) = j_1 \otimes j_2$.
Thus $j$ is supported on $Q^-N=U^-MU$ and 
\begin{equation*}
j(u^-mu) = \psi(u) \pi_M(m) (j_1\otimes j_2),
\quad
u^-\in U^-, \
m\in M,\ 
u\in U,
\end{equation*}
and more generally one has $j(gx)=\psi(x)j(g)$ for all $g\in G$ and all $x\in N$.

For~$i=1,2$,
fix $\xi_i\in\Hom_{\H_{n_i}}(\pi_i,1)$ such that $\xi_i(j_i)\neq0$.
Define an $H_n$-invariant linear form~$\xi$ on the space of $\pi$ by 
\begin{equation*}
\xi(f) = \sum\limits_{h \in Q^-\cap\H_n\backslash H_n}
\langle f(h), \xi_1\otimes\xi_2 \rangle
\end{equation*}
where $\langle \cdot , \cdot \rangle$ is the duality bracket.
It is well-defined thanks to \eqref{kakiIw1} and \eqref{kakiIw2}.
Let us now~compu\-te~$\xi(j)$.
Thanks to \eqref{kakiIw1}, \eqref{kakiIw2} and \eqref{kaki1},
we have
\begin{eqnarray*}
\xi(j) &=& \sum\limits_{h \in Q^-\cap H_n \backslash H_n}
\langle j(h), \xi_1 \otimes \xi_2 \rangle \\ &=&
\sum\limits_{u \in U\cap H_n}
\langle j(u), \xi_1 \otimes\xi_2 \rangle \\ &=&
| U\cap H_n | \cdot \xi_1(j_1) \cdot \xi_2(j_2)
\end{eqnarray*}
which is non-zero.
This finishes the proof.
\end{proof}

\begin{coro}
\label{Aubusson}
Let $n_1,\dots,n_r$ be positive integers of sum $n$.
For $i=1,\dots,r$, 
let $\pi_i$ be some $H_{n_i}$-distinguished cuspidal representation
of $\G_{n_i}$. 
Then $\pi_1\times\dots\times\pi_r$ is of class $\Cc(H_n)$.
\end{coro}

\begin{proof}
This follows from Corollary \ref{CpsiCusp} and Lemma \ref{cuspsatcpsi2}.
\end{proof}

\subsection{}
\label{parthenon}
\label{EUE}
\label{union}

We now prove Proposition \ref{minelli3},
which will end the proof of our main Theorem \ref{CONJNRV}.
We will actual\-ly~prove the more general following result.

\begin{prop}
\label{minelli4}
Let $\pi$ be an $H_n$-distingui\-shed cuspidal representation of $\G_n$
and~let~$\mu$~be~an $H_m$-distinguished supercuspidal representation
of $\G_m$ for some $m<n$.
Then
\begin{equation*}
\g(\pi,\mu,\psi) = \left\{
\begin{array}{rl}
{\rm sgn}(\pi) & \text{in the Levi case with $m=1$,} \\
1 & \text{otherwise.}
\end{array}
\right.
\end{equation*}
\end{prop}

Let us first explain how to deduce Proposition \ref{minelli3} from Proposition 
\ref{minelli4}.
Let $\pi$ be a non-super\-cuspidal~cus\-pi\-dal representation of
$G_n$,
and $\rho$ be the supercuspidal repre\-sentation of~$G_k$~associa\-ted~with it. 
Suppose that $\pi$~is~${H}_n$-distingui\-shed~but has no
${H}_n$-dis\-tinguished lift.
By Lemma \ref{distselfdual},
it is $\s$-self-dual.
By uniqueness of its supercuspidal support,
$\rho$ is $\s$-self-dual as well.~By~Lem\-ma \ref{distparitek},
$\rho$ is distinguished,
unless we are in the Levi case and $\rho$ is the 
character of order $2$ of $\k^\times$.~But the Levi case with $k=1$
is excluded by \cite{NRV25} Lemma 6.10(2.b),
which says that $\pi$ has a~dis\-tin\-gui\-shed $\qlb$-lift in that case. 
To deduce Proposition \ref{minelli3},
it thus remains to apply Proposition~\ref{minelli4}~to the distinguished
representation $\mu=\rho^\vee$.

Let us prove Proposition \ref{minelli4} in the exceptional
case where $H_n$ is a Levi subgroup and $m=1$. 
In~this case,
$\mu$ is the trivial character of $\k^\times$.
We thus have $\g(\pi,\mu,\psi) = {\rm sgn}(\pi)$ thanks to~Pro\-po\-sitions 
\ref{RSGJ} 
and \ref{valueGJdist}.
It now remains to prove Proposition \ref{minelli4}:
\begin{itemize}
\item either in the Galois case (in which case $\psi$ is trivial on $\k_0$),
\item or in the Levi case with $m\neq1$ (in which case $m$ is even by
Lemma \ref{distparitek}).
\end{itemize}
Let $\pi$ be an ${H}_n$-distingui\-shed cus\-pi\-dal representation of $G_n$.
Denote by $1^a$, 
for any integer~$a\>1$,
the induced repre\-sen\-ta\-tion
$1\times\dots\times1$~whe\-re~the trivial character $1$ of $\k^\times$ occurs 
$a$ times.
It~is~of Whittaker type and satisfies $\Cc(H_{a})$~by Corollary \ref{Aubusson}.
By Corollary \ref{CpsiCusp},
the representation $\mu$ satisfies $\Cc (H_m)$.
By~Co\-rollary \ref{Aubusson} again, 
the re\-pre\-sentation $\tau=\mu \times 1^{n-1-m}$~satisfies~$\Cc(H_{n-1})$.
By Lemma \ref{basicCpsilemma},
the representations $1^{n-1}$ and $\tau$
are both $H_{n-1}$-special. 
It follows from Proposi\-tions \ref{mult} and \ref{Lemma1RSconstant} that
\begin{eqnarray*}
\g(\pi,\mu,\psi) &=& \g(\pi,\mu \times 1^{n-1-m},\psi) 
\cdot \g(\pi,1^{n-1-m},\psi)^{-1} \\ 
                       &=& \g(\pi,1^{n-1},\psi) \cdot \g(\pi,1,\psi)^{-n+1+m} \\
&=& \g(\pi,1,\psi)^{m} 
\end{eqnarray*}
which is equal to $\g(\pi,\psi)^{m}$ by Proposition \ref{RSGJ}.
Proposition \ref{minelli4} follows automatically~from~Pro\-po\-si\-tion
\ref{valueGJdist},
together with the fact that $m$ is even in the Levi case. 
Theorem \ref{CONJNRV} is proven.

\begin{coro}
For any $H_n$-distinguished cuspidal representation $\pi$ of $G_n$, 
one has
\begin{equation*}
c(\pi,\psi) = \left\{ 
\begin{array}{ll}
1 & \text{in the Galois case,} \\ 
{\rm sgn}(\pi)^{n-1} & \text{in the Levi case.}
\end{array}\right.
\end{equation*}
\end{coro}

\begin{proof}
By Proposition \ref{Lemma1RSconstant}, we have
$c(\pi,\psi) = \g(\pi,1^{n-1},\psi) = \g(\pi,1,\psi)^{n-1}$.
The result~fol\-lows from Propositions \ref{RSGJ} and \ref{valueGJdist}.
\end{proof}

\begin{rema}
\label{minelli2}
Let $\pi$ be a non-supercuspidal cuspidal representation of $\G_n$,
with~associa\-ted supercuspidal representation $\rho$ of degree $k$. 
Suppose that $\pi$~is $H_n$-distingui\-shed.
Then
\begin{equation*}
\g(\pi,\rho^\vee,\psi) = \left\{
\begin{array}{rl}
-1 & \text{in the Levi case with $k=1$,} \\
1 & \textit{otherwise.}
\end{array}
\right.
\end{equation*} 
The only case which is not covered by Proposition \ref{minelli4} is the 
Levi case where $\rho$ is the character of~order $2$ of $\k^\times$.
In~this case, we have $\g(\pi,\rho^\vee,\psi) = \g(\pi',1,\psi)$
where $\pi'=\pi\rho$ is the unique~cus\-pi\-dal~subquotient
of $1^n$.
We then have $\g(\pi',1,\psi)={\rm sgn}(\pi')$ thanks to Propositions
\ref{RSGJ}, \ref{valueGJdist},
and the result now follows from \cite{NRV25} Lemma 6.12.
\end{rema}

\subsection{}

We give the following corollaries to Theorem \ref{CONJNRV}.

\begin{coro}
\label{pafpifnonramifie}
Let $\ke/\kf$ be a quadratic extension of finite fields~of
characteristic~$p$ and~$\pi$~be a non-supercuspidal,
cuspi\-dal representation of $\GL_n(\k)$.
The following assertions are equivalent.
\begin{enumerate}
\item $\pi$ is $\GL_n(\k_0)$-dis\-tin\-guished. 
\item $\pi$ has a $\GL_n(\k_0)$-dis\-tingui\-shed cuspidal lift to $\qlb$. 
\item $\pi$ has a $\s$-self-dual cuspidal lift to $\qlb$. 
\item $\pi$ is $\s$-self-dual, $n$ is odd and $e_0$ is even.
\end{enumerate}
\end{coro}

\begin{coro}
\label{pafpiframifie}
Let $\ke$ be a finite field of characteristic $p$
and $\pi$ be~a non-supercuspidal,~cuspi\-dal 
representation of $\GL_{n}(\k)$ with $n=2m$.
The following assertions are equivalent.
\begin{enumerate}
\item $\pi$ is $\GL_{m}(\k)\times\GL_{m}(\k)$-dis\-tin\-guished.
\item $\pi$ has a $\GL_{m}(\k)\times\GL_{m}(\k)$-dis\-tin\-guished
  cuspidal lift to $\qlb$. 
\item $\pi$ has a self-dual cuspidal lift to $\qlb$. 
\item $\pi$ is self-dual and either $r$, $n/e$ are odd, or $r=n$. 
\end{enumerate}
\end{coro}

\begin{proof}
We will prove the two corollaries at the same time. 
In both cases,
we have the chain~of implications
$(4) \Rightarrow (3) \Rightarrow (2) \Rightarrow (1)$.
It thus remains to prove that (1) implies (4) in both cases.

Fix a quadratic extension $F/F_0$ of $p$-adic fields such that 
$\ke$, $\kf$ are the re\-si\-due fields of $\F$,~$\F_0$
respectively,
where we write $\k_0=\k$~in the Levi case.
Fix a uni\-formizer $\w$ of $\F$ such that $\w^{e_{F/F_0}}$
is~a uni\-formizer of $\F_0$,
where $e_{F/F_0}$ denotes the ramification order of $F/F_0$.
Thus $\s(\w)\in\{-\w,\w\}$ in any case.
Let $\boldsymbol{J}^0$ denote 
\begin{itemize}
\item the maximal compact open subgroup $\GL_n(\oo)$ if $F/F_0$
  is unramified,
\item 
the conjugate of $\GL_n(\oo)$ by 
\begin{equation*}
{\rm diag}(\w,\dots,\w,1,\dots,1) \in \GL_n(\F)
\end{equation*}
where $\w$ occurs $m$ times,
if $F/F_0$ is ramified.
\end{itemize}
Let $\boldsymbol{J}^1$ be the normal maximal pro-$p$-subgroup of 
$\boldsymbol{J}^0$ and $\boldsymbol{J}$ be its normalizer in $\GL_n(\F)$.
The~na\-tu\-ral group~iso\-morphism 
$\boldsymbol{J}^0/\boldsymbol{J}^1\simeq\GL_n(\ke)$
transports the action~of $\s\in\Gal(\F/\F_0)$ on 
$\boldsymbol{J}^0/\boldsymbol{J}^1$~to 
\begin{itemize}
\item the action~of $\s\in\Gal(\k/\k_0)$ if $F/F_0$ is unramified,
\item
the adjoint action of
\begin{equation*}
\begin{pmatrix}
-{\rm id}_{m} & 0 \\ 0 & {\rm id}_{m}
\end{pmatrix} \in \GL_n(\ke)
\end{equation*}
on $\GL_n(\ke)$ if $F/F_0$ is ramified.
\end{itemize}

Let us consider $\pi$ as a representation of $\boldsymbol{J}^0$~by~in\-flation,
and extend it to a representation $\bl$~of 
$\boldsymbol{J}=\F^\times\boldsymbol{J}^0$ by demanding that the central character
of $\bl$ at a uni\-formizer $\w$ of $\F$ is~equal~to~$1$~if $F/F_0$ is
unramified,
and to ${\rm sgn}(\pi)$ if $F/F_0$ is ramified.
In any case,
$(\boldsymbol{J},\bl)$ is a generic $\s$-self-dual level $0$ type in the sense 
of \cite{NRV25} Definition 4.31.
Inducing $\bl$ to $\GL_n(\F)$, 
we get~a~cuspidal~(irre\-du\-cible) repre\-sen\-tation $\pi_F$ of level $0$
of $\GL_n(\F)$. 
By \cite{NRV25} Theorem 4.45,
the fact~that $\pi$ is~dis\-tinguished im\-plies that $\pi_F$ is 
$\GL_n(\F_0)$-distinguished. 
It follows from Theorem \ref{CONJNRV}~that~$\pi_F$
has a $\GL_n(\F_0)$-distinguished cuspidal $\qlb$-lift,
and from \cite{NRV25} Lemma 6.5 that
$\pi$ has an $H$-distinguished cuspidal $\qlb$-lift.
\end{proof}

\section{The second main result ($\ell=2$)}

We now assume that the field $R$ has characteristic $\ell=2$.

\subsection{}

Here is our second main theorem.

\begin{theo}
\label{MAINTHM2}
Suppose that $\ell=2$.
A cuspidal $R$-representation of $\GL_n(\F)$ is $\GL_n(\F_0)$-dis\-tin\-gui\-shed 
if and only if it is $\s$-self-dual. 
\end{theo}

Note that:
\begin{itemize}
\item 
any $H$-distinguished cuspidal representation of $G$ is $\s$-self-dual
(\cite{VSANT19} Theorem 4.1),
\item
Theorem \ref{MAINTHM2} holds for supercuspidal representations
(\cite{VSANT19} Theorem 10.8).
\end{itemize}

To prove Theorem \ref{MAINTHM2},
it thus suffices to prove that any $\s$-self-dual 
non-su\-per\-cuspi\-dal,~cuspidal representation of $G$ is $H$-distinguished.
As in the proof of Proposition \ref{lamiel},
it follows from~\cite{NRV25}~Pro\-position 4.40 that the proof of Theorem 
\ref{MAINTHM2} reduces to the level $0$ case.
As $\ell=2$,
the central~char\-acter of any $\s$-self-dual irreducible representation of $G$
is trivial on $F_0^\times$.
Applying \cite{NRV25}~Theorem 4.45,
we are thus reduced to proving the following result. 

\begin{theo}
\label{arcueil}
Let $\s$ be one of the two following involutions of $\GL_n(\k)$: 
\begin{itemize}
\item 
either $\k$ is a quadratic extension of $\k_0$ and $\s$ is the 
non-trivial~auto\-mor\-phism of $\k/\k_0$, 
\item
or $n=2m$ for some integer $m\>1$ and $\s$ is the inner auto\-mor\-phism
of conjugacy by~the~dia\-go\-nal element
$\d_{n}={\rm diag}(1,-1,1,\dots,-1)$, 
\end{itemize}
and let $H$ be the subgroup of $\s$-fixed elements of $\GL_n(\k)$.
Then any~$\s$-self-dual cuspi\-dal~represen\-tation of $\GL_n(\k)$ is
$\H$-distinguished. 
\end{theo}

\begin{rema}
\label{soitditenpassant2}
When $\ell\neq2$, there are $\s$-self-dual 
cuspidal $\overline{\FF}_\ell$-repre\-sen\-ta\-tions of $\GL_2(F)$ which are 
not distinguished
(see Remark \ref{haydee}).
\end{rema}

\subsection{}
\label{eugeniedanglars}

Let $G_n=\GL_n(\k)$ for some $n\>1$,
let $\s$ be one of the involutions of $G_n$ of Theorem \ref{arcueil}~and
let $H_n$ be the subgroup of $\s$-fixed points of $G_n$.
The following lemma follows from \cite{MSc} 2.4.

\begin{lemm}
Let $\pi$ be a cuspidal representation of $G_n$.
\begin{enumerate}
\item 
There are a unique divisor $r$ of $n$
and a supercuspidal representation $\rho$ of $G_{n/r}$,
unique~up~to isomorphism,
such that $\pi$ occurs as a subquotient of the induced
representation $\rho^{\times r}$.
\item
One has $r=2^a$ for some $a\>0$.
\item
$\pi$ is the unique cuspidal subquotient
of $\rho^{\times r}$ and it occurs with multiplicity $1$.
\end{enumerate} 
In this situation, we will write $\pi=\st_a(\rho)$.
\end{lemm}

Let $\pi$ be a cuspidal representation of $G_n$.

\begin{lemm}
The induced~repre\-sentation $\pi\times\pi$ is indecomposable of length $3$.
It has~a~uni\-que cuspidal subquotient, which occurs with multiplicity $1$,
and a non-cuspidal irreducible subquotient~occur\-ring with multiplicity $2$. 
\end{lemm}

\begin{proof}
First,
$\pi\times\pi$ contains a uni\-que generic subquotient, occurring with 
multiplicity $1$.~Wri\-ting $\pi$ under the form $\st_a(\rho)$,
it also contains the cuspidal (thus generic) representation $\st_{a+1}(\rho)$.
This proves the second assertion.
Its unique non-zero proper Jacquet module has length $2$
and is made of $\pi\otimes\pi$ with multiplicity $2$.
It thus has at most two non-cuspidal irreducible subquotients. 
Since its cuspidal subquotient can't appear as a subrepresentation
nor a quotient, 
this length~has to be~$3$ and $\pi\times\pi$ is indecomposable.
By \cite{MSc} Théorème 5.3, 
the number of its non-iso\-morphic ir\-re\-ducible subquotients is equal to
the number of partitions of $2$, which is $2$. 
Thus the non-cuspidal irreducible subquotients are isomorphic. 
\end{proof}

Let us de\-no\-te by $\st_1(\pi)$ the cuspidal subquotient of $\pi\times\pi$
and by $\tau$ its unique irreducible~quo\-tient,
which is isomorphic to its unique irreducible subrepresentation.
In this paragraph, we~re\-du\-ce the proof of Theorem \ref{arcueil} to that of the 
the following proposition.

\begin{prop}
\label{AndreaCavalcanti}
Let $\pi$ be an $H_n$-distinguished cuspidal representation of $G_n$ for
some $n\>1$.
Then the cuspidal representation $\st_1(\pi)$ of $G_{2n}$ is
$H_{2n}$-distinguished.
\end{prop}

Let us explain how this proposition implies Theorem \ref{arcueil}.
Let $\pi$ be a $\s$-self-dual cuspidal~repre\-sen\-tation of $G_n$ for some $n\>1$.
It is of the form $\st_{a}(\rho)$ for an integer~$a\>0$ 
and a supercuspidal representation $\rho$.
By uniqueness, $\rho$ is $\s$-self-dual,
thus it is distinguished by \cite{VSANT19} Theorem 10.8. 
On the other hand,
we denote by $\pi_i$ the cuspidal representation $\st_{i}(\rho)$.
One has $\pi_{i+1}=\st_1(\pi_i)$ for all $i\>0$.
By induction,
thanks to Proposition \ref{AndreaCavalcanti} and the supercuspidal case,
$\pi_i$ is distin\-gui\-shed for all $i\>0$.
In particular, $\pi_a=\pi$ is distinguished.

It thus remains to prove Proposition \ref{AndreaCavalcanti}.~We
will prove it by contradiction,
assuming that the cuspidal representation
$\st_1(\pi)$ is not $H_{2n}$-distinguished.

\subsection{}

Let $\pi$ be an $H_n$-distinguished cuspidal representation of $G_n$ for
some $n\>1$,
and let $\tau$ be the unique~irreduci\-ble quotient of $\pi\times\pi$. 
We first prove the following multiplicity $1$ result.

\begin{prop}
\label{onedim}
The space $\Hom_{H_{2n}}(\tau,1)$ has dimension at most $1$. 
\end{prop}

\begin{proof}
In the Galois case, 
this is \cite{VSANT19} Remark 4.3.
We will thus assume from now on that~we are in the Levi case. 
For simplicity, we wri\-te~$H=H_{2n}$, $P=P_{2n}$ and $G=G_{2n}$.
Since $\Hom_{H}(\tau,1)$ is~con\-tained in
$\Hom_{P\cap H}(\tau,1)$, it~suf\-fices to prove 
\begin{equation*} 
\dim \Hom_{P\cap H}(\tau,1) \< 1.
\end{equation*}
Our argument is inspired from \cite{AnandMatringe} Corollary 4.4. 
We will use the theory of derivatives for $R$-re\-pre\-sentations of general
linear groups over $\k$,
for which we refer to \cite{Vigbook}~III.1.~(Unlike~\cite{Vigbook}, 
we will use the usual notation $\Phi^+$, $\Phi^-$, $\Psi^+$, $\Psi^-$,
the definition of which can be found~in \cite{AnandMatringe}~Section 4
for instance.)
It will be convenient to define $G_0=H_0$ to be the trivial group.
We will need~the following property (see \cite{AnandMatringe} Proposition
4.3 in the Galois case for $R=\CC$).

\begin{lemm}
\label{psiprop}
\begin{enumerate}
\item 
For any $i\>2$ and any representation $\kappa$ of $P_{i-1}$,
one has an isomorphism of $R$-vector spaces
\begin{equation}
\label{iso1}
\Hom_{P_i\cap H_i}(\Phi^+\kappa,1) \simeq \Hom_{P_{i-1}\cap H_{i-1}}(\kappa,1).
\end{equation}
\item
For any $i\>1$ and any representation $\mu$ of $G_{i-1}$,
one has an equality
\begin{equation}
\label{iso2}
\Hom_{P_i\cap H_i}(\Psi^+\mu,1) = \Hom_{G_{i-1}\cap H_{i-1}}(\mu,1).
\end{equation}
\end{enumerate} 
\end{lemm}

\begin{proof}
Let us embed $G_{i-1}$ in $G_i$ via $g \mapsto {\rm diag}(g,1)$.
Then \eqref{iso2} immediately follows from~the fact that
$P_i\cap H_i = H_{i-1}(U_i\cap H_i)$ (this is \eqref{kakiMi}),
and \eqref{iso1} is given by \cite{VSANT19} Lemma 2.10. 
\end{proof}

First,
it follows from \cite{Vigbook} III.1.3 that
\begin{equation}
\label{majdim}
\dim \Hom_{P\cap H}(\tau,1) \<
\sum \limits_{i=1}^{2m} \dim \Hom_{P\cap H}((\Phi^+)^{i-1}\Psi^+\tau^{(i)},1).
\end{equation}
We are going to prove that $\tau^{(i)}$ is zero,
unless $i=n$ in which case $\tau^{(n)}=\pi$.
It will follow that
\begin{equation*}
\label{majdim2}
\dim \Hom_{P\cap H}(\tau,1) \<
\dim \Hom_{P\cap H}((\Phi^+)^{i-1}\Psi^+\pi,1) =
\dim \Hom_{H_n}(\pi,1) = 1
\end{equation*}
where the first equality follows from Lemma \ref{psiprop}
and the second one from Lemma \ref{malebranche}. 

Recall that the semi-simplification of $\pi\times\pi$ is equal to
$\st_1(\pi) + 2\tau$.
Since the derivative functors are exact,
the semi-simplification of $(\pi\times\pi)^{(i)}$ is equal to
$\st_1(\pi)^{(i)} + 2\tau^{(i)}$ for all $i\>0$.~The Leib\-niz~rule
(\cite{Vigbook} III.1.10)
together with the fact that $\pi$ is cuspidal imply that
$(\pi\times\pi)^{(i)}$ is zero,~un\-less $i\in\{n,2n\}$,
and that $(\pi\times\pi)^{(n)}$ has length $2$ with subquotient
$\pi$ occurring twice.
For $i=2n$,~the
derivative $\tau^{(2n)}$ is zero since $\tau$ is not generic.
For $i=n$,
$\st_1(\pi)^{(n)}$ is zero since $\st_1(\pi)$ is cuspidal,
thus $\tau^{(n)}=\pi$.
\end{proof}

\subsection{}
\label{strategik}

We are now going to construct two independent non-zero $H_{2n}$-invariant 
linear forms on~the induced representation
$\pi\times\pi$ vanishing on the socle of $\pi\times\pi$. 
Assuming that the cuspidal~repre\-sen\-tation
$\st_1(\pi)$ is not $H_{2n}$-distinguished,
they will thus vanish on the maximal proper subrepre\-sentation of 
$\pi\times\pi$,
thus inducing two independent
non-zero $H_{2n}$-invariant linear forms~on $\tau$~which
will contradict Proposition \ref{onedim}.
For simplicity, we will wri\-te~$H=H_{2n}$, $P=P_{2n}$ and $G=G_{2n}$.

Let $Q$ be the standard~para\-bo\-lic subgroup of $G$ of size $(n,n)$
together with its standard Levi subgroup $M$ and unipotent radical $U$.
The endomorphism algebra $\End_G(\pi\times\pi)$ contains~an~ele\-ment
$T$ defined by
\begin{equation*}
Tf(g) = \sum \limits_{u \in U} A f(sug)
\end{equation*}
for $f\in\pi\times\pi$ and $g\in G$,
where $A$ is the isomorphism $v \otimes w \mapsto w \otimes v$ on the tensor
square of the space of $\pi$.
Denoting by $s$ the element
\begin{equation*}
s = \begin{pmatrix}0&1_n\\ 1_n&0\end{pmatrix} \in G,
\end{equation*}
this isomorphism is the unique intertwiner between the representation 
$\pi\otimes\pi$ of $M$ and its conjugate by $s$ such that $A^2$ is the 
identity. 
The element $T$ satisfies $T^2=1$, or equivalently $(T+1)^2=0$,
and the image of $T+1$ is the socle of $\pi\times\pi$.
A linear form $\La$ on the space of $\pi\times\pi$ thus vanishes on the socle of 
$\pi\times\pi$ if and only if $\La\circ T=\La$.

Note that $Q$, $U$ and $M$ are $\s$-stable and that $s\in H$.
We will need the following lemma.

\begin{lemm}
\label{baeckaoffa}
\begin{enumerate}
\item (Galois case)
For $i=0,\dots,n$, define
\begin{equation*}
\g_i =
\begin{pmatrix}
1_{i} & & & \\ & 0 & 1_{n-i} & \\ & 1_{n-i} & 0 & \\ & & & 1_{i}
\end{pmatrix}.
\end{equation*}
For any choice of $x_i \in G$ such that $\s(x_i^{\phantom{1}})x_i^{-1}=\g_i$, 
the set $\{x_0,\dots,x_n\}$ is a set of representatives of $(Q,H)$-double
cosets of $G$.
\item (Levi case)
For $i,j\in\{0,\dots,n\}$ such that $0\<i+j\<n$, define
\begin{equation*}
\g_{i,j} =
\d_{2n} 
\begin{pmatrix}
  1_i & & & & \\
  & -1_j & & & \\
  & & 0 & 1_{n-i-j} & & \\
  & & 1_{n-i-j} & 0 & & \\
  & & & & 1_j & \\
  & & & & & -1_i
\end{pmatrix}. 
\end{equation*}
For any choice of $x_{i,j} \in G$ such that $\s(x_{i,j}^{\phantom{1}})x_{i,j}^{-1}=\g_{i,j}$, 
the set $\{x_{i,j},0\<i+j\<n\}$ is a set~of~re\-presentatives of $(Q,H)$-double
cosets of $G$.
\end{enumerate}
\end{lemm}

\begin{proof}
Let $V$ be the $\k$-vector space $\k^{2n}$ equipped with its canonical 
basis $(e_1,\dots,e_{2n})$
and~$W_\circ$ be the subspace of $V$ spanned by $(e_1,\dots,e_{n})$.
The map $g \mapsto g^{-1}W_\circ$ induces a bijection between 
the $(Q,H)$-double cosets of $G$ and
the $H$-orbits of subspaces of dimension $n$ of $V$.

Let us start with the Galois case.
The space $V$ is equipped with the action of $\s$ componentwise.
By \cite{MatringeIMRN11} Section 3,
these $H$-orbits are in bijection with $\{0,\dots,n\}$ through
$W \mapsto \dim (W \cap \s(W))$.
For any $g\in G$,
one has
\begin{equation*}
\dim (g^{-1}W_\circ \cap \s(g^{-1}W_\circ))
  = \dim (\s(g)g^{-1}W_\circ \cap W_\circ).
\end{equation*}
It thus follows that $x_i^{-1}W^{\phantom{1}}_\circ$ corresponds to $i$
for all $i\in\{0,\dots,n\}$,
and the result follows. 
See \cite{MatringeIMRN11} Proposition 3.7.

Let us now consider the Levi case. 
The space $V$ is equipped with the natural action of
$\d=\d_{2n}$. By \cite{MatringeCRELLE15} Section 3,
the $H$-orbits of subspaces of dimension $n$ of $V$
are in bijection with the set of pairs of
integers $(i,j)\in\{0,\dots,n\}^2$ such that $i+j\<n$ through
\begin{equation*}
W \mapsto (\dim\Ker (\d-1\ |\ W \cap \d W), \dim\Ker (\d+1\ |\ W \cap \d W)).
\end{equation*}
For any $g\in G$ and any sign $\e\in\{-1,1\}$,
one has
\begin{equation*}
\dim\Ker (\d-\varepsilon\cdot1\ |\ g^{-1}W_\circ \cap \d g^{-1} W_\circ)
= \dim \Ker (g\d g^{-1}-\varepsilon\cdot1 \ |\ g\d g^{-1}W_\circ \cap W_\circ).
\end{equation*}
It then follows that
{$x_{i,j}^{-1}W^{\phantom{1}}_\circ$ corresponds to
the pair $(i,j)$ for all $i,j\in\{0,\dots,n\}$ such that $i+j\<n$.}
The result follows. 
See \cite{MatringeCRELLE15} Proposition 3.2.
\end{proof}

\subsection{}

Let us define a first $H$-invariant linear form on $\pi\times\pi$.
Since $\pi$ is $H_n$-distinguished,
we fix~a non-zero linear form $\l\in\Hom_{H_n }(\pi,1)$
and write $\La_{\pi}^0 = \l\otimes\l\in \Hom_{H_n \times H_n}(\pi\otimes\pi,1)$.
It defines~an $H$-invariant linear form $\La^0$ on $\pi\times\pi$ given by
\begin{equation}
\label{defLa0}
  \La^0(f) = \sum \limits_{h \in Q\cap\H\backslash\H}
  \langle f(h), \La_{\pi}^0\rangle.
\end{equation}
This linear form is non-zero.
Indeed, 
if $f$ is supported on $Q$ and $f(1)=v \otimes v$ where $v$ is a
vector in the space of $\pi$ such that $\l(v)\neq0$, one has
$\La^0(f)=\langle f(1), \La_{\pi}^0\rangle=\l(v)^2\neq0$.
Let us prove that $\La^0$ vanishes on the socle of $\pi\times\pi$.

\begin{prop}
One has $\La^0\circ T = \La^0$. 
\end{prop}

\begin{proof}
Since $\La^0$ is $H$-invariant,
it suffices to prove that $\La^0(Tf) = \La^0(f)$ for any
$f \in \pi\times\pi$~sup\-ported on $Qg$,
where $g$ ranges over a set of representatives of $(Q,H)$-double
cosets of $G$.
  
First,
an $h\in Q\cap\H\backslash\H$ contributes~to the sum in \eqref{defLa0}
if and only if $h\in\H\cap Qg$,
which~is~non-empty if and only if $g\in QH$,
in which case we may assume that $g=1$,
thus $\La^0(f) = \langle f(1), \La_{\pi}^0\rangle$.

Since $\La_{\pi}^0\circ A=\La_{\pi}^0$,
one has
\begin{equation*}
\La^0(Tf) = \sum \limits_{h \in Q\cap\H\backslash\H} \sum \limits_{u \in U} 
\langle f(suh), \La_{\pi}^0\rangle.
\end{equation*}
We are now going to determine the $h\in Q\cap\H\backslash H$ and
$u\in U$ such that $suh \in Qg$.
Set $\g=\s(g)g^{-1}$.
By  Lem\-ma~\ref{baeckaoffa},
we may assume that $Q\g Q$ is equal to
$Q\g_i Q$ for some parameter $i$ (in the Galois case)
or $Q\g_{i,j}Q$ for some $i,j$ (in the Levi case). 
If $suh \in Qg$,
then $s\s(u)u^{-1}s \in Q\g Q \cap U^-$~with
\begin{equation}
\label{QgaQ}
Q\g Q = Q
\begin{pmatrix}
1_{k} & & & \\ & 0 & 1_{n-k} & \\ & 1_{n-k} & 0 & \\ & & & 1_{k}
\end{pmatrix}
Q
\end{equation}
for some $k\in\{0,\dots,n\}$.
(More precisely,
one has $k=i$ in the Galois case if $\g=\g_i$,
and $k=i+j$ in the Levi case if $\g=\g_{i,j}$.)
This double coset has a non-empty intersection with $U^-$ if and
only if $k=n$,
in which case this intersection is reduced to $\{1\}$. 
It follows that $s\s(u)u^{-1}s=1$,
whence $u\in U\cap H$.
Since $s\in H$,
the intersection $(U\cap H)sQg \cap H$ is non-empty if and
only if $g\in QH$,~in which case
{we may assume that $g=1$.}
One then has
\begin{equation*}
h\in (U\cap H)sQ \cap H=
(U\cap H)s(Q \cap H) = (Q\cap H)s(U \cap H).
\end{equation*}
Setting $h=sv$ with $v\in U\cap H$,
one has $susv\in Q$ if and only if $sus\in Q\cap U^-=\{1\}$.
Thus $u=1$ and
\begin{equation*}
\La^0(Tf) = \sum \limits_{v \in U\cap\H} 
\langle f(v), \La_{\pi}^0\rangle
= |U\cap H| \cdot \langle f(1), \La_{\pi}^0\rangle
= \La^0(f)
\end{equation*}
since $ |U\cap H| \equiv 1$ mod $2$.
\end{proof}

\subsection{}

Let us now define a second $H$-invariant linear form on $\pi\times\pi$.
In order to treat the Galois~and Levi cases uniformly,
we set $\d=1$ in the Galois case and $\d=\d_{2n}$ in the Levi case.
We also set
\begin{equation*}
\t = 
\left\{ 
\begin{array}{ll}
\s & \text{in the Galois case}, \\ 
{\rm id} & \text{in the Levi case}.
\end{array}\right.
\end{equation*}
Fix an $x\in G$ such that $\s(x)x^{-1}=\d s$. 
Note that $\d s$ is $\g_{0}$ in the Galois case
and $\g_{0,0}$ in the Levi case,
and that
\begin{equation*}
{}^xH\cap Q = {}^xH\cap M = \{ (g,\t(g)) \in M \ |\ g \in G_n \}.
\end{equation*}
Fix an isomorphism between $\pi$ and $\pi^{\vee\t}$,
or equivalently a non-zero linear form 
$\La_{\pi}^1 : \pi \otimes\pi \to R$ such that
$\La_{\pi}^1(\t(g)v,gw)=\La_{\pi}^1(v,w)$.
It defines an $H$-invariant linear form $\La^1 $ on $\pi\times\pi$
given by
\begin{equation}
\label{defLa1}
  \La^1(f) = \sum \limits_{h \in M^x\cap\H\backslash\H}
  \langle f(xh), \La_{\pi}^1 \rangle.
\end{equation}
It is non-zero.
Indeed, 
if $f$ is supported on $Qx$ and $f(x) \notin \Ker(\La_{\pi}^1)$,
then $\La^1(f)=\langle f(x), \La_{\pi}^1\rangle\neq0$.
Let us prove that $\La^1$ vanishes on the socle of $\pi\times\pi$.

\begin{prop}
One has $\La^1\circ T = \La^1$. 
\end{prop}

\begin{proof}
Since $\La^1$ is $H$-invariant,
it suffices to prove that $\La^1(Tf) = \La^1(f)$ for any
$f \in \pi\times\pi$~sup\-ported on $Qg$
where $g$ ranges over a set of representatives of $(Q,H)$-double
cosets of $G$.

We set $\g=\s(g)g^{-1}$.
By  Lemma~\ref{baeckaoffa},
we may assume that $Q\g Q$ is equal to
$Q\g_i Q$ for some~$i$~(in the Galois case)
or $Q\g_{i,j}Q$ for some $i,j$ (in the Levi case).
In any case, we can write it as in \eqref{QgaQ} for some $k$.

First,
an $h\in M^x\cap\H\backslash\H$ contributes to the sum in \eqref{defLa1}
if and only if $xh\in x\H\cap Qg$.
If~this~in\-ter\-section is non-empty,
then,
by applying the map $y \mapsto \s(y)y^{-1}$,
we get $\d s\in Q\g Q$,
which can happen if and only if $k=0$,
thus $g\in QxH$,
in which case we may assume that $g=x$.
One then has $h\in H\cap Q^x=H\cap M^x$
and $\La^1(f) = \langle f(x), \La_{\pi}^1\rangle$.

By uniqueness of the isomorphism $\pi\simeq\pi^{\vee\t}$
up to a non-zero scalar,
the linear form $\La_{\pi}^1$ that we have fixed is unique
up to a non-zero scalar. 
Thus $\La_{\pi}^1\circ A=\mu\cdot\La_{\pi}^1$ for some
$\mu\in\R^\times$.
Since $A^2$ is the identity and $R$ has characteristic $2$,
it follows that $\mu=1$,
that is, $\La_{\pi}^1\circ A=\La_{\pi}^1$.
One has
\begin{equation*}
\La^1(Tf) = \sum \limits_{h \in M^x\cap\H\backslash\H} \sum \limits_{u \in U} 
\langle f(suxh), \La_{\pi}^1\rangle.
\end{equation*}
We are now going to determine the $h\in M^x\cap\H\backslash H$ and
$u\in U$ such that $suxh \in Qg$. 
If $suxh \in Qg$,
then $s\s(u)\d su^{-1}s \in Q\g Q$.
Let us write 
\begin{equation*}
s\s(u)\d su^{-1}s = s
\begin{pmatrix} 1&\s(t)\\&1 \end{pmatrix} \d s 
\begin{pmatrix} 1&-t\\&1 \end{pmatrix} s =
\d \begin{pmatrix} -t&1\\ 1-\t(t)t&\t(t) \end{pmatrix} 
\quad\text{with}\quad
u = \begin{pmatrix} 1&t\\&1 \end{pmatrix}
\end{equation*}
for some $t\in\Mat_n(\k)$.
Such an element is in $Q\g Q$ if and only if $k=0$,
in which case we may~assu\-me that $g=x$ and 
$\g=\d s$.
Then $QsQ$ is made of those $a\in\Mat_2(\Mat_n(\k))$
such that $a_{2,1}$ is invertible.
Let $t\in\Mat_n(\k)$ and suppose~that $1-\t(t)t \in G_n$.
Then $suxh \in Qx$ implies that
$xhx^{-1} \in QsQ \cap {}^xH$.
Since $ {}^xH$ is the subgroup of elements of $G$ fixed by the involution
$g \mapsto s\t(g)s$, an easy calculation shows that it 
is made of the matrices of $G$ of the form
\begin{equation}
\label{formedexhx}
\begin{pmatrix} a&b\\ \t(b)&\t(a) \end{pmatrix},
\quad a,b \in\Mat_n(\k)
\end{equation}
thus $QsQ \cap {}^xH$ corresponds to those matrices with $b$ invertible.
Since we consider $h$ up to~mul\-ti\-plication by $M^x\cap H$ on the left,
and since
\begin{equation*}
M \cap {}^xH = \begin{pmatrix} a& \\ &\t(a) \end{pmatrix},
\quad a \in\G_n,
\end{equation*}
we may assume that $xhx^{-1}$ is of the form \eqref{formedexhx} with $b=1_n$.
Thus
\begin{equation*}
suxhx^{-1} = \begin{pmatrix} 1 & -\t(t) \\ t+a & 1+t\t(a) \end{pmatrix}
\end{equation*}
belongs to $Q$ if and only if $a=-t$. 
We thus get
\begin{eqnarray*}
\La^1(Tf) &=& \sum \limits_{t \in {\it\Omega}} 
\langle f\left(\begin{pmatrix} 1 & -\t(t) \\ 0 & 1-t\t(t) \end{pmatrix}x\right), \La_{\pi}^1\rangle\\
&=& \sum \limits_{t \in {\it\Omega}} 
\langle \begin{pmatrix} 1 & \\ & 1-t\t(t) \end{pmatrix} \cdot f(x), \La_{\pi}^1\rangle
\end{eqnarray*}
where ${\it\Omega}$ denotes the set of $t\in\Mat_n(\k)$ such that
$1-\t(t)t \in G_n$.
Note that $0\in{\it\Omega}$ and that $t$, $-t$ have the same image by 
$t\mapsto 1-t\t(t)$,
thus ${\it\Omega}$ is stable by $t\mapsto -t$.
Fix a set ${\it\Omega}^+$ of representatives of the cosets $\{t,-t\}$ in 
${\it\Omega}-\{0\}$. 
We then have
\begin{equation*}
\La^1(Tf) = \langle f(x), \La_{\pi}^1\rangle
+ 2 \sum \limits_{t \in {\it\Omega}^+} 
\langle \begin{pmatrix} 1 & \\ & 1-t\t(t) \end{pmatrix} \cdot f(x),
\La_{\pi}^1\rangle = \langle f(x), \La_{\pi}^1\rangle.
\end{equation*}
This finishes the proof.
\end{proof}

\subsection{}

In conclusion,
we have defined two independent non-zero $H$-invariant 
linear forms $\La^0$, $\La^1$ on $\pi\times\pi$ vanishing
on its socle.
If $\st_1(\pi)$ were not $H$-distinguished,
these linear forms would~in\-du\-ce two independent non-zero $H$-invariant 
linear forms on the irreducible quotient $\tau$ of $\pi\times\pi$,
thus contradicting Proposition \ref{onedim}.
This proves Proposition \ref{AndreaCavalcanti}
and finishes the proof of our~se\-cond main Theorem \ref{MAINTHM2},
as explained in \S\ref{eugeniedanglars}.

\providecommand{\bysame}{\leavevmode ---\ }

\end{document}